\documentclass[reqno,12pt]{amsart}
\usepackage[english]{babel}
\usepackage{amsmath}

\usepackage{amssymb}
\usepackage{graphicx}
\usepackage{epstopdf}

\usepackage{subfigure}
\usepackage{enumitem}

\usepackage{amsaddr}

\newcommand{\bpr}{\begin{trivlist} \item[]{\bf Proof. }}
\newcommand{\epr}{\hspace*{\fill} $\qed$\end{trivlist}}

\newcommand{\be}{\begin{eqnarray}}
\newcommand{\ee}{\end{eqnarray}}
\newcommand{\ba}{\begin{align}}
\newcommand{\ea}{\end{align}}
\newcommand{\bi}{\begin{itemize}}
\newcommand{\ei}{\end{itemize}}

\newcommand{\secref}[1]{Section~\ref{sec:#1}}

\newcommand{\seclab}[1]{\label{sec:#1}}
\newcommand{\eqlab}[1]{\label{eq:#1}}
\renewcommand{\eqref}[1]{(\ref{eq:#1})}

\newcommand{\figref}[1]{Fig.~\ref{fig:#1}}
\newcommand{\figlab}[1]{\label{fig:#1}}
\newcommand{\propref}[1]{Proposition~\ref{proposition:#1}}
\newcommand{\proplab}[1]{\label{proposition:#1}}
\newcommand{\lemmaref}[1]{Lemma~\ref{lemma:#1}}
\newcommand{\lemmalab}[1]{\label{lemma:#1}}

\newcommand{\thmref}[1]{Theorem~\ref{theorem:#1}}
\newcommand{\thmlab}[1]{\label{theorem:#1}}

\newtheorem{theorem}{Theorem}[section]
\newtheorem{proposition}[theorem]{Proposition}

\newtheorem{lemma}[theorem]{Lemma}

\newtheorem{remark}[theorem]{Remark}

\numberwithin{equation}{section}
\usepackage{fullpage}
\usepackage{color,comment,xcolor}
\definecolor{orange}{RGB}{255,127,0}
\newcommand\response[1]{{\color{black}{#1}}}
\begin{document}
\title{The number of limit cycles for regularized piecewise polynomial systems is unbounded}
\author{R. Huzak}
\address{Hasselt University, Campus Diepenbeek, Agoralaan Gebouw D, 3590 Diepenbeek, Belgium}
\author{K. Uldall Kristiansen}
\address{Department of Applied Mathematics and Computer Science, 
Technical University of Denmark, 
2800 Kgs. Lyngby, 
Denmark }

%\author{S. Jelbart \and K. U. Kristiansen \and P. Szmolyan \and M. Wechselberger} 
%\date {}
%\date\today

%\WM{
%\note{WM: my suggestions for possible journals are:  Dynamical Systems (Taylor \& Francis), Journal of Dynamics and Differential equations (Springer).}
%}
 \begin{abstract}
 In this paper, we extend the slow divergence-integral from slow-fast systems, due to De Maesschalck, Dumortier and Roussarie, to smooth systems that limit onto piecewise smooth ones as $\epsilon\rightarrow 0$. In slow-fast systems, the slow divergence-integral  is an integral of the divergence along a canard cycle with respect to the slow time and it has proven very useful in obtaining good lower and upper bounds of limit cycles in planar polynomial systems. In this paper, our slow divergence-integral is based upon integration along a generalized canard cycle for a piecewise smooth two-fold bifurcation (of type visible-invisible called $VI_3$). We use this framework to show that the number of limit cycles in regularized piecewise smooth polynomial systems is unbounded.  

\bigskip
\smallskip

\noindent \textbf{keywords.} Slow divergence-integral, canards, piecewise smooth systems, two-folds, GSPT
 \end{abstract}
 \maketitle
% \tableofcontents
 \section{Introduction}
 In this paper, we consider smooth systems of the form
 \begin{align}
  \dot z &=Z(z,\phi(h(z)\epsilon^{-1})),\eqlab{ztZ}
 \end{align}
 for $z\in \mathbb R^n$, $0<\epsilon\ll 1$ and
 where $h:\mathbb R^n\rightarrow \mathbb R$ is regular, $\phi$ is a regularization function:
\begin{align}
 \phi'(s)>0\mbox{ for all }  s\in \mathbb R,
\quad 
 \phi(s) \rightarrow \begin{cases}
                      1 &\text{for}\,\,s\rightarrow \infty\\
                      0 &\text{for}\,\,s\rightarrow -\infty
                     \end{cases}\eqlab{phi1st}
\end{align}
and where $Z$ is affine in its second component:
\begin{align}\eqlab{affine}
 Z(z,p) = Z_+(z) p + Z_-(z)(1-p).
\end{align}
These systems have recently received a great deal of attention \cite{bossolini2017a,bossolini2020a,jelbart2021c,jelbart2021b,kosiuk2016a,kristiansen2018a,kristiansen2020a,kristiansen2015a,uldall2021a}. The motivation is three-fold. Firstly, in the limit $\epsilon\rightarrow 0$ the system \eqref{ztZ} becomes piecewise smooth (PWS)
\begin{align}
 \dot z &=\begin{cases}
           Z_+(z)  \text{ for }h(z)>0,\\
           Z_-(z)  \text{ for }h(z)<0,
          \end{cases}\eqlab{pws}
\end{align}
with $\Sigma:=\{z:h(z)=0\}$ being a discontinuity/switching manifold, \response{see \figref{pws}}.

\begin{figure}[h!]
\begin{center}
{\includegraphics[width=.65\textwidth]{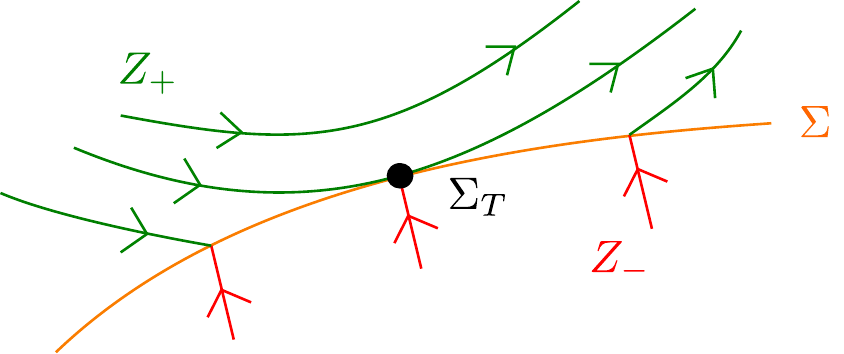}}
% \subfigure[]{\includegraphics[width=.495\textwidth]{layerReducedBlowup.pdf}}
% \input{posusts}
\end{center}
\caption{\response{A planar piecewise smooth system \eqref{pws}, having $\Sigma$ as a switching manifold. Here regular orbits of $Z_+$ and $Z_-$ reach $\Sigma$ in finite time. The point $\Sigma_T$ is a tangency point of $Z_+$ (a visible fold if the tangency is quadratic). In the present case, where $Z_-$ is transverse to $\Sigma$ it divides $\Sigma$ into sliding (to the left of $\Sigma_T$) where $Z_\pm$ are in opposition relative to $\Sigma$ and crossing (to the right of $\Sigma_T$) where $Z_\pm$} \response{point in the same direction relative to $\Sigma$. The situation is different if $Z_-$ also has a fold at $\Sigma_T$ (called a two-fold).} \response{Then there can be sliding (stable and unstable) on each side of $\Sigma_T$,} \response{see \figref{twofold123}.}}
\figlab{pws}
\end{figure}

For $0<\epsilon\ll 1$, the system \eqref{ztZ} is therefore a regularized PWS system \cite{Bernardo08,filippov1988differential}. 
The reason for restricting to \eqref{affine} is that in this case, one can show \cite{Sotomayor96} that the singular limit system is a Filippov system \cite{Bernardo08,filippov1988differential}. Lately, there has been a growing interest in understanding how PWS phenomena (folds, grazing, boundary equilibria,... \cite{Kuznetsov2003}) unfold in the smooth version \cite{jelbart2021b,jelbart2021c,kristiansen2018a,kristiansen2020a,kristiansen2015a}. For this purpose methods from Geometric Singular Perturbation Theory (GSPT) and blowup have been refined to deal with resolving the special singular limit of  \eqref{ztZ} \cite{kristiansen2018a,kristiansen2020a}. 
Finally, the interest in systems of the form \eqref{ztZ} is also motivated by applications. For example, in biology switches \cite{Bernardo08,uldall2021a} are frequently modeled by functions \eqref{phi1st} and friction is also inherently piecewise smooth \cite{berger2002a}.  

Mathematically, piecewise smooth system has also received a great deal of attention over the past decades. Starting from the groundbreaking work of Filippov \cite{filippov1988differential} and Utkin \cite{Utkin77}, there was an effort to extend Peixito's program of structural stability to PWS systems \cite{broucke2001a,Sotomayor96}. Subsequently, there has been a focus on characterizing and interpreting the lack of uniqueness of solutions in PWS systems \cite{jeffrey2018a}.

Parallel to this effort, there has been an attempt to bound the number limit cycles in PWS systems in the plane where $n=2$. In contrast to the \response{smooth linear} setting, limit cycles can exist for piecewise linear systems and J. Llibre and co-workers have obtained upper bounds for a number of cases \cite{esteban2021a,li2021a,llibre2013a}. 
Of course, the interest in bounding the number of limit cycles, comes from Hilbert's 16th problem \cite{li2003a} which seeks to bound the number of limit cycles of polynomial systems:
\begin{equation}\eqlab{xyPQ}
\begin{aligned}
 \dot x &=P_N(x,y),\\
 \dot y &=Q_N(x,y),
\end{aligned}
\end{equation}
with $P_N$ and $Q_N$ of fixed degree $N$. Hilbert's 16th problem remains unsolved to this day. Whereas general progress has been made on $N=2$ \cite{artes,Gavrilov,Mardesic,dumortier1996a,dumortier1994a,DDRR,Ren,RR15} and on Smale's version of the problem  where \eqref{xyPQ} is restricted to classical Li\'{e}nard type: $P_N(x,y) = y-p_N(x),\,Q_N=-x$, see \cite{caubergh2012a,CZLL,LMP,smale}, there has been an emphasis on obtaining good lower bounds on the number of limit cycles (see \cite{ChrisL,Valery,Han1,LMT} and references therein). Following the work of De Maesschalck, Dumortier and Roussarie, see \cite{DM-entryexit,DM,dumortier2011a,dumortier_1996,dumortier2001a}, a key tool in this effort has been the \textit{slow divergence-integral} from slow-fast systems and canard theory; in particular, the roots of the slow divergence-integral provide candidates for limit cycles. For example, using this tool good lower bounds on the number of limit cycles in Li\'{e}nard equations can be found (see \cite{DDMoreLC,SDICLE1,DPR,HuDe,lvarez2020a}).

\subsection{Main result}\seclab{mainres}
In this paper, we work at the interface of these research fields. In particular, we consider \eqref{ztZ} with $n=2$, put $z=(x,y)$ and restrict attention to the case $h(z)=y$ so that the switching manifold is $\Sigma= \{(x,y):y=0\}$ and then ask the following question:

\textit{Does there exist polynomial vector-fields $Z_\pm$ such that the number of limit cycles of $Z$ is unbounded?}
% \end{align*}

We prove that this is in fact true, even for quadratic vector-fields. More precisely we prove the following. 
\begin{theorem}\thmlab{mainthm}
 There exists a quadratic vector-field $Z_+(\cdot,\lambda)$ and a linear vector-field $Z_-(\cdot,\lambda)$, depending smoothly on a parameter $\lambda\in\mathbb{R}$, such that the following holds true in a compact domain $U$:
 
 For every $k\in \mathbb N$ there exist: (a) $\epsilon_k>0$, (b) a regularization function $\phi_k:\mathbb R
 \rightarrow \mathbb R$, and (c) a continuous function $\lambda_c^k:[0,\epsilon_k[\rightarrow \mathbb R$ such that the regularized vector-field:
 \begin{align*}
  Z(z) = Z_+(z,\lambda_c^k(\epsilon))\phi_k(y\epsilon^{-1})+Z_-(z,\lambda_c^k(\epsilon))(1-\phi_k(y\epsilon^{-1})),
 \end{align*}
has at least $k$ limit cycles contained in $U$ for all $\epsilon\in ]0,\epsilon_k[$. 
\end{theorem}
\response{We give examples of $Z_+$ and $Z_-$ later on, see \eqref{Zp1}\eqref{Zn1}.}
We emphasize that the unboundedness of limit cycles stems from the regularization and not from the vector-fields $Z_\pm$. \response{We use smooth regularization functions in order to find an unbounded number of limit cycles. It is known that boundedness of limit cycles is closely related to the notion of o-minimality in function spaces (see e.g. \cite{O-minimal}). Our smoothings are
taken from a family that does not have this o-minimality property. From this viewpoint it is not surprising that we find that the number of limit cycles is unbounded. }

\response{At the same time, \thmref{mainthm} also illustrates a certain degree of deficiency with smoothing piecewise smooth systems (since the result may depend upon how we regularize). On the other hand, there are other complementary results, see \cite{bossolini2020a,jelbart2021b,jelbart2021c,kristiansen2018a}, that show that smoothing play little role (at least on a macroscopic-level, i.e. at $\mathcal O(1)$) for different types of PWS singularities and bifurcations. In \cite{kristiansen2018a} for example, it was shown that the regularization of the visible-invisible fold in $\mathbb R^3$, with $\Sigma$ being two-dimensional, is independent of the smoothing function. In fact, for the system in \thmref{mainthm} it is also only in an exponentially small parameter regime that a different number of limit cycles can be realized for different regularization functions. }

To prove \thmref{mainthm}, we will follow the approach of \cite{DM-entryexit} and use a slow divergence-integral. But seeing that our system is nonsmooth  (as opposed to slow-fast) in the singular limit $\epsilon\rightarrow 0$ we will first have to develop this framework within the setting of \eqref{ztZ}. For slow-fast systems, the slow divergence-integral is defined along a canard trajectory, i.e. along a folded critical manifold with an equilibrium at the fold in such way that the reduced problem goes from the attracting sheet to the repelling one with nonzero speed. In the setting of \eqref{ztZ}, our slow divergence-integral will be based upon the PWS two-fold bifurcation \cite{bonet-reves2018a,kristiansen2015a}, which is reminiscent of the standard canard \cite{dumortier_1996,krupa_relaxation_2001}. In particular, $Z_\pm$ in \thmref{mainthm} will be chosen so that the PWS system has a two-fold bifurcation (of type visible-invisible called $VI_3$ \cite{Kuznetsov2003}).
\response{Proposition \ref{prop-diffmap} then describes the structure of the difference map near the associated canard-like limit periodic sets (see \secref{diffmap}).}

\response{Proposition \ref{prop-diffmap} is not only relevant and important for proving \thmref{mainthm}, but also for studying bifurcations of limit cycles inside such visible-invisible two-folds (see Remark \ref{remark-Prop-important}). This proposition is therefore also one of our main results, but we delay the detailed statement to later sections after having introduced the two-fold bifurcation model (see \secref{twofold}).}

\response{Our approach for constructing an unbounded number of limit cycles, does not work for the piecewise linear case. It would be interesting to study the linear case more carefully in future work.}
\subsection{Overview}
 The paper is organized as follows: In \secref{twofold}, we define a planar PWS two-fold and revisit some results from \cite{bonet-reves2018a,kristiansen2015a} on canards of \eqref{ztZ} for $0<\epsilon\ll 1$. Next in \secref{slowdiv} we define the slow divergence-integral and prove that simple roots of this function lead to hyperbolic limit cycles (\thmref{mainslowdiv}). In the proof of \thmref{mainslowdiv}, \response{we describe the difference map in terms of the slow divergence integral in Proposition \ref{prop-diffmap}. For the proof of this statement, we also use
 Appendix \ref{appendix-0} and Appendix \ref{appendix}.} %Moreover, based upon our findings for the slow divergence-integral,  we present in \secref{slowdiv} some new results -- independent of \thmref{mainthm} -- on the number of limit cycles bifurcating from the two-fold, see \thmref{mainrh1}--\thmref{tm-catastrophy}. 
 In \secref{final} we then prove \thmref{mainthm}, using \thmref{mainslowdiv}, see also \thmref{thmhere}, and \response{finally in  \secref{num} we illustrate our approach with numerical examples}.

\section{The two-fold bifurcation}\seclab{twofold}
We consider \eqref{ztZ} with $h(z)=y$:
 \begin{align}
  \dot z &=Z(z,\phi(y\epsilon^{-2}),\lambda),\eqlab{zZ}
 \end{align}
for $z=(x,y)\in \mathbb R^{2}$. In comparison with \eqref{ztZ} we have also included $\lambda\sim \lambda_0\in\mathbb{R}$ as an additional unfolding parameter. Notice also that we write $\epsilon^{-2}$ in \eqref{zZ} rather than just $\epsilon^{-1}$, since this will be convenient later on (see \secref{cylblowup}). The basic assumption is that the right hand side $Z$ is smooth in each entry (in this paper, by ``smooth" we mean differentiable of class $C^\infty$). In particular we suppose that it is affine in the second component, i.e.,
\begin{align*}
 Z(z,p,\lambda) = Z_+(z,\lambda)p + Z_-(z,\lambda)(1-p),
\end{align*}
where $Z_\pm =(X_\pm,Y_\pm)$ are smooth in $(z,\lambda)$. The function $\phi:\mathbb{R}\to \mathbb{R}$ is a smooth sigmoidal function satisfying the following assumptions:
\begin{enumerate}[label=({A}{{\arabic*}})]
%  \item 
% \end{enumerate}
% 
 \item \label{assA} The function $\phi$ has the following asymptotics when $s\to\pm\infty$: \begin{align*}
 \phi(s)\rightarrow \begin{cases}
                     1 & \text{for}\quad  s\rightarrow \infty,\\
                     0 & \text{for}\quad s\rightarrow -\infty.
                    \end{cases}
\end{align*}
\item \label{assB} The function $\phi$ is strictly monotone, i.e.,  $\phi'(s)>0$ for all $s\in \mathbb R$.
\item \label{assC} The function $\phi$ is smooth at $\pm \infty$ in the following sense: Each of the functions
\begin{align*}
 \phi_+(s):=\begin{cases}
             1 & \text{for}\quad s=0,\\
             \phi(s^{-1}) & \text{for}\quad s>0,
            \end{cases},\quad
            \phi_-(s):=\begin{cases}
             \phi(-s^{-1}) & \text{for}\quad s>0,\\
             0 & \text{for}\quad s=0,
            \end{cases}
\end{align*}
are smooth at $s=0$. %\setcounter{nameOfYourChoice}{\value{enumi}}
% \begin{document}
% 
% \begin{enumerate}
% \item Foo
\end{enumerate}

% Foo

% \begin{enumerate}
% 
% \end{enumerate}
% \begin{remark}To be able to obtain a complete desingularized picture, one would also have to suppose that $\phi_\pm$, defined in Assumption C, have a $k$th-order derivative which is nonzero at $s=0$, e.g. such that $\phi_+(s) = 1-s^{k} \beta (s)$ with $\beta(0)>0$.
% \end{remark}
% \begin{remark}
% \end{remark}
\smallskip

By assumption \ref{assA}, the system \eqref{zZ} is piecewise smooth (PWS) in the limit $\epsilon\rightarrow 0$:
\begin{align}
 \dot z = \begin{cases}
           Z_+(z,\lambda) &\text{for}\quad y>0,\\
           Z_-(z,\lambda) &\text{for}\quad y<0,
          \end{cases}\eqlab{zZpm}
\end{align}
the set $\Sigma$ defined by $(x,0)$ being the discontinuity set/switching manifold, for each $\lambda\sim \lambda_0$. 
In fact, from assumption \ref{assC} we have that \eqref{zZ} is a regular perturbation of $Z_+$ or $Z_-$ outside any fixed neighborhood of $y=0$. In particular:
\begin{lemma}\label{lemma-regular}
Suppose that \response{there is a smallest $k\in \mathbb N$} such that $\phi_+^{(k)}(0)\ne 0$. Then within $y\ge c$, with $c>0$ fixed
\begin{align*}
 Z = Z_++\mathcal O(\epsilon^{2k}),
\end{align*}
smoothly and uniformly with respect to $\epsilon\rightarrow 0$. 
 \end{lemma}
A similar result obviously holds for $Z_-$ (in terms of $\phi_-^{(k)}\ne 0$).
% Let $h(x,y)=y$ such that $\Sigma = \{(x,y)\vert h(x,y)=0\}$. 
In PWS theory \cite{Bernardo08} we divide $\Sigma$ into different subsets $\Sigma_{cr}(\lambda)$, $\Sigma_{sl}(\lambda)$ and $\Sigma_T(\lambda)$, each depending upon on $\lambda$, which are defined as follows:
% The ``crossing'' subset $\Sigma_{cr}\subset \Sigma$ is the set of all points $q=(x,0)$ where
%  \begin{align*}
%   Y_+(q,\lambda)Y_-(q,\lambda)>0.
%  \end{align*}
% \end{definition}
% The following concepts -- where $\lambda$ plays little role and is therefore suppressed in $X,Y$ and $Z$ -- are then classical in PWS theory (see e.g. \cite{kristiansen2015a}):
\begin{itemize}
 \item[(1)] The subset $\Sigma_{cr}(\lambda)\subset \Sigma$ consisting of all points $q=(x,0)$ where
 \begin{align*}
  Y_+(q,\lambda)Y_-(q,\lambda)>0,
 \end{align*}
 is called ``crossing''.
\item[(2)] The subset $\Sigma_{sl}(\lambda)\subset \Sigma$ consisting of all points $q=(x,0)$ where
\begin{align*}
  Y_+(q,\lambda)Y_-(q,\lambda)<0.
 \end{align*}
 is called ``sliding''. It is said to be stable (resp. unstable) if $Y_+<0$ and $Y_->0$, (resp. $Y_+>0$ and $Y_-<0$).
 \item[(3)] The subset $\Sigma_T(\lambda)\subset \Sigma$ where either $Y_+(q,\lambda)=0$ or $Y_-(q,\lambda)=0$ is called the PWS singularities.
 \end{itemize}
%   we have the following classical result: 
It is well-known \cite{Sotomayor96}, that once assumption \ref{assB} holds, sliding for \eqref{zZpm} implies existence of an invariant manifold for \eqref{zZ}. 
\begin{theorem}\thmlab{slidinghmtsd}
Suppose that \ref{assA} and \ref{assB} hold true and that the PWS system \eqref{zZpm} has stable/unstable sliding along some subset $\Sigma_{sl}\subset \Sigma$, i.e. $Y_+(x,0,\lambda)Y_-(x,0,\lambda)<0$ for $(x,0)\in \Sigma_{sl}$. \response{Let $I$ be a compact interval so that $I\times \{0\}\subset \Sigma_{sl}$.} Then for all $0<\epsilon\ll 1$, there is a locally invariant manifold of \eqref{zZ} with foliation by stable/unstable fibers, respectively, of the following graph form $y=\epsilon^2 h(x,\epsilon^2)$, \response{$x\in I$}. The reduced dynamics for $\epsilon\rightarrow 0$ on this manifold is given by:
\begin{align}
 \dot x &=X_{sl}(x,\lambda) := X_+(x,0,\lambda)p+X_-(x,0,\lambda)(1-p),\eqlab{Xsl}
\end{align}
% \begin{equation}
%      \label{reduced-global}
%      \dot{x}=X_{sl}(x,\lambda):=\frac{X_-Y_+-X_+Y_-}{Y_+-Y_-}(x,0,\lambda), \ x\ne 0.
%  \end{equation}
where $p=p(x)\in ]0,1[$ solves $Y_+(x,0,\lambda)p+Y_-(x,0,\lambda)(1-p)=0$.
\end{theorem}
\begin{proof}
 The proof is elementary so we include it. Define $y_2$ by $y=\epsilon^2 y_2$. Then 
 \begin{equation}\eqlab{divergence}
 \begin{aligned}
  x' &=\epsilon^2 X(x,\epsilon^2 y_2,\phi(y_2),\lambda),\\
  y_2' &=Y(x,\epsilon^2 y_2,\phi(y_2),\lambda),
 \end{aligned}
 \end{equation}
i.e. a slow-fast system with $Y(x,0,\phi(y_2),\lambda)=Y_+(x,0,\lambda)\phi(y_2)+Y_-(x,0,\lambda)(1-\phi(y_2))=0$ defining a critical manifold for $\epsilon=0$. 
Linearization around any point on this manifold for $\epsilon=0$ produces a single nontrivial eigenvalue $\left(Y_+(x,0,\lambda)-Y_-(x,0,\lambda)\right)\phi'(y_2)$ which is nonzero since $Y_+Y_-<0$ and since \ref{assB} holds. In fact, its sign is only determined by $Y_+$ and $Y_-$. Hence the critical manifold, which takes a graph form
\begin{align*}
 y_2 = \phi^{-1} \left(\frac{-Y_-}{Y_+-Y_-}(x,0,\lambda)\right),\, (x,0)\in \Sigma_{sl}(\lambda),
\end{align*}
is hyperbolic and attracting/repelling whenever the associated sliding is stable/unstable. The result therefore follows by Fenichel's theory \cite{fen3}.
\end{proof}

By plugging the expression for $$p(x,\lambda)=\frac{-Y_-}{Y_+-Y_-}(x,0,\lambda),$$ into \eqref{Xsl}, we may write $X_{sl}$ as 
\begin{align}
X_{sl}(x,\lambda)=\frac{\text{det}\,Z}{Y_+-Y_-}(x,0,\lambda),\eqlab{Xsl2}
\end{align}
where
\begin{align*}\text{det}\,Z (x,0,\lambda) := (X_-Y_+-X_+Y_-)(x,0,\lambda)\end{align*}
The vector-field \eqref{Xsl2} is known as the Filippov sliding vector-field \cite{filippov1988differential} and PWS systems with this vector-field prescribed on $\Sigma_{sl}$ are called Filippov systems.
% The vector-field

\subsection{Folds} 
 Clearly, $\Sigma=\Sigma_{sl}(\lambda)\cup \Sigma_T(\lambda) \cup \Sigma_{cr}(\lambda)$ for each $\lambda$. %Moreover, if $q\in \Sigma_T(\lambda_0)$ then  $q\notin \Sigma_T(\lambda)$ in general for $\lambda\ne \lambda_0$. 
 We further classify the points in $\Sigma_T$ as follows (see also \cite{Bernardo08}):
 \begin{itemize}
  \item[(4)] A point $q\in \Sigma_T(\lambda)$ is a fold point from ``above'' if the orbit of $Z_+(\cdot,\lambda)$ through $q$ has a quadratic tangency with $\Sigma$ at $q$. In terms of Lie-derivatives $Z_\pm (h)(\cdot,\lambda) := \nabla h \cdot Z_\pm(\cdot,\lambda)$, with $h(x,y)=y$, the last condition becomes: 
  \begin{align*}
  \begin{cases}
  Z_+(q,\lambda )&\ne 0,\\ 
  Z_+(h)(q,\lambda)&=0, \\
  Z_+^2(h)(q,\lambda)&\ne 0.\end{cases}
  \end{align*} We define a fold point from ``below'' in terms of $Z_-$ in a similar way. 
  \item[(5)] A fold point $q\in \Sigma_T(\lambda)$ from ``above'' is said to be visible, if the orbit of $Z_+(\cdot,\lambda)$ through $q$ is contained within $y>0$ in neighborhood of $q$. It is said to be invisible otherwise. In terms of Lie-derivatives, we clearly have $Z_+^2(h)(q,\lambda)>0$ iff $q$ satisfying $Z_+(q,\lambda)\ne 0$, $Z_+(h)(q,\lambda)=0$ is visible. Fold points from below are classified in a similar way. In particular, $Z_-^2(h)(z)<0$ iff $q$ satisfying $Z_-(q,\lambda)\ne 0$, $Z_-(h)(q,\lambda)=0$ is visible.
%  \item A fold point from ``below'' is visible if $Z_-^2(f)(z)<0$ invisible otherwise.
 \end{itemize}
 Fold points  that are only PWS singularities on one side of $\Sigma$ are persistent by the implicit function theorem, in the following sense: If $\Sigma_T(\lambda_0)$ consists of a fold point $q(\lambda_0)$ (from above or below), then $\Sigma_T(\lambda)$ also consists of a fold point $q(\lambda)$ (from above or below, respectively) for any $\lambda\sim \lambda_0$. In fact, $q(\lambda)$ then also depends smoothly on $\lambda\sim \lambda_0$. 
 
 \subsection{Two-folds} \seclab{section-2-fold}
 Now, we finally arrive at the concept of two-folds in PWS systems, which will play the role of a canard point in our analysis of \eqref{zZ}.
 \begin{itemize}
 \item[(6)] A two-fold $q\in \Sigma_T(\lambda)$ is a point with quadratic tangencies from above \textit{and} from below. In terms of Lie-derivatives we have: 
  \begin{align}\eqlab{twofoldcond}
  \begin{cases}
  Z_\pm (q,\lambda )&\ne 0,\\ 
  Z_\pm (h)(q,\lambda)&=0, \\
  Z_\pm ^2(h)(q,\lambda)&\ne 0,\end{cases}
  \end{align}
  with these equations understood to hold for \textit{both} $\pm$.
%  
%  $Z_\pm(q,\lambda)\ne 0,Z_\pm (h)(q,\lambda)= 0$, $Z_\pm^2(h)(q,\lambda)\ne 0$. 
 \item[(7)] A two-fold is said to be visible-visible, visible-invisible, invisible-invisible according to the ``visibility'' of the fold from above and below, respectively, see item (5) above. 
\end{itemize}

The three distinct cases are illustrated in \figref{twofold123}. The further details depend on the direction of the flow. In fact, according to \cite{Kuznetsov2003} there are 7 cases, two visible-visible cases (called $VV_{1,2}$), three visible-invisible cases (called $VI_{1-3}$) and two invisible-invisible cases $II_{1,2}$). We refer to \cite{Kuznetsov2003} as well as \cite{bonet-reves2018a,kristiansen2015a} for further details here. They will not be needed in the present manuscript. \smallskip

In contrast to a fold, a two-fold is a co-dimension one (PWS) bifurcation \cite{bonet-reves2018a}. Consequently, if $q\in \Sigma_T(\lambda_0)$ is a two-fold then generically there is a neighborhood $U$ of $q$ such that $Z_\pm(\cdot,\lambda)$ for $\lambda\ne \lambda_0$, $\lambda\sim \lambda_0$ does not have any two-folds in $U$. Upon writing $Z_\pm(\cdot,\lambda)=Z_\pm (\cdot,\lambda_0)+(\lambda-\lambda_0)\widetilde Z_\pm(\cdot)+\mathcal O((\lambda-\lambda_0)^2)$, \cite[Theorem 2.6]{bonet-reves2018a} showed that the unfolding is versal if
\begin{align}
\widetilde Y_- Y_+' \ne  \widetilde Y_+ Y'_-.
% \frac{\widetilde Y_-}{Y_-'}\ne \frac{\widetilde Y_+}{Y_+'}
\eqlab{versal}
\end{align}
 Here we denote by $()'$ \textit{the partial derivative with respect to $x$, a convention we will continue to adopt in the following.}

\begin{figure}[h!]
\begin{center}
\subfigure[]{\includegraphics[width=.3\textwidth]{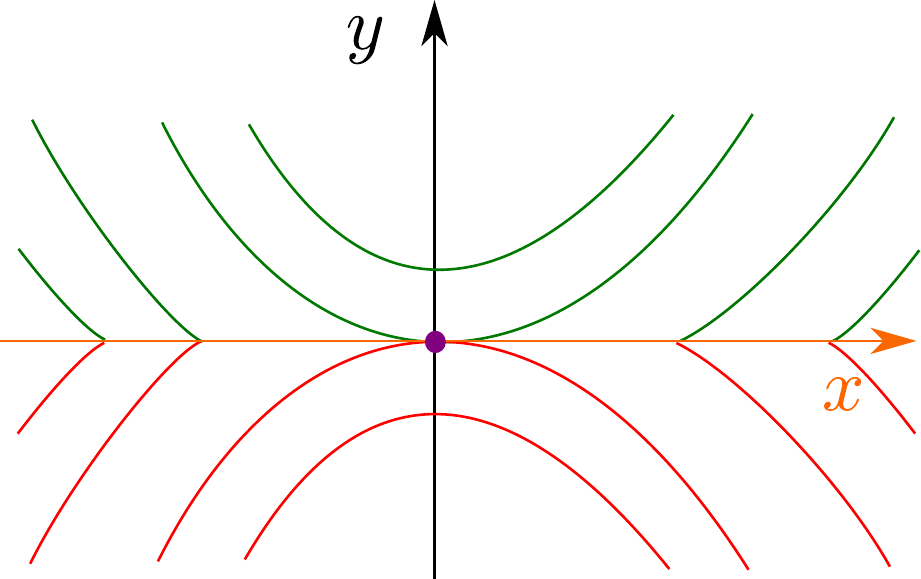}}
\subfigure[]{\includegraphics[width=.3\textwidth]{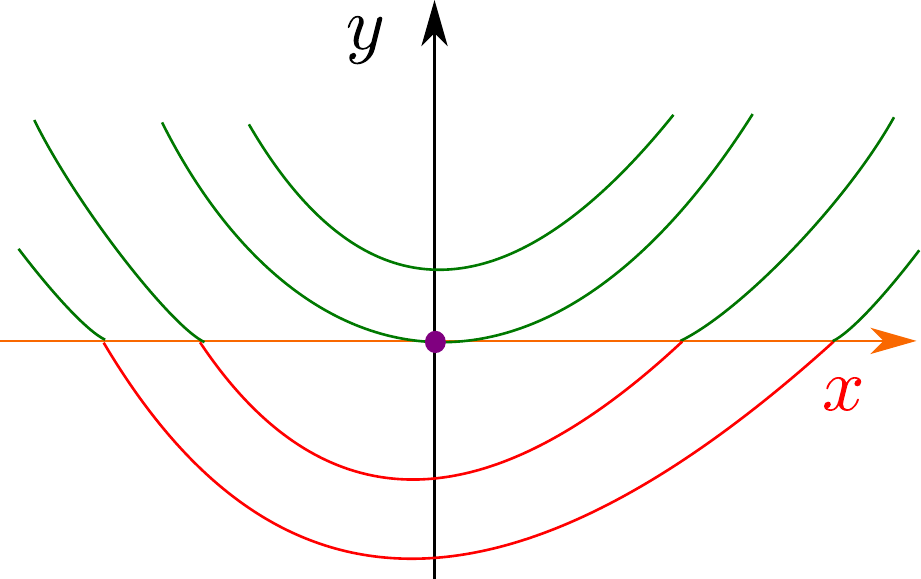}}
\subfigure[]{\includegraphics[width=.3\textwidth]{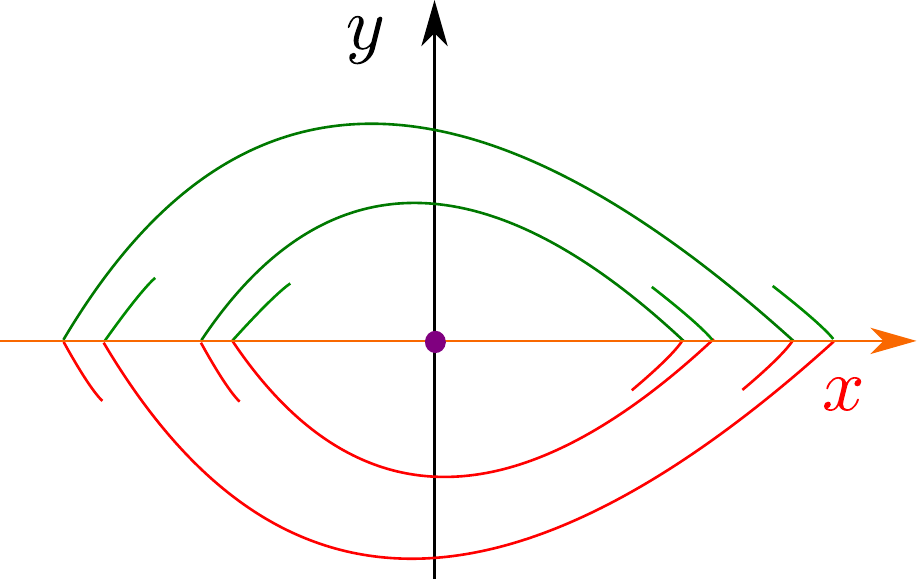}}
% \subfigure[]{\includegraphics[width=.495\textwidth]{layerReducedBlowup.pdf}}
% \input{posusts}
\end{center}
\caption{The three two-folds: visible-visible (a), visible-invisible (b) and invisible-invisible (c). We have deliberately not put arrows on the orbits of $Z_-$ (red) and $Z_+$ (green), because $\Sigma_{sl}$ and $\Sigma_{cr}$ depend on this direction. Notice $\Sigma$ (orange) is the $x$-axis in all figures. Following \cite{Kuznetsov2003} there are 7 cases, two visible-visible cases (called $VV_{1,2}$), three visible-invisible cases (called $VI_{1-3}$) and two invisible-invisible cases $II_{1,2}$). The case $VI_3$, which will be our main focus, is illustrated separately in \figref{twofoldVI3}.}
\figlab{twofold123}
\end{figure}

In the present paper, we will focus on the visible-invisible two-fold. In this case, \cite[Lemma 2.8]{bonet-reves2018a} shows that if $q$ is a visible-invisible two-fold for $\lambda=\lambda_0$, then locally
\begin{align*}
 \Sigma(\lambda_0) = \Sigma_{sl}(\lambda_0)\cup \{q\}
\end{align*}
whenever 
\begin{align}\eqlab{XpXncond}
X_+(q,\lambda_0)X_-(q,\lambda_0)<0.
\end{align}Consequently, $X_{sl}(x,\lambda_0)$ is in this case locally defined for all points on $\Sigma(\lambda_0)$ except $q$ (see \thmref{slidinghmtsd}). Notice in particular from the form \eqref{Xsl2} that $X_{sl}(x,\lambda_0)$ has a ``0/0'' at the two-fold. However, by \eqref{twofoldcond} and \eqref{XpXncond} we also have that
\begin{align}\eqlab{detZp}
   Y_+'-Y_-'\ne 0,
\end{align}
at $(q,\lambda_0)$, and consequently from \eqref{Xsl2} we see that $X_{sl}(x,\lambda_0)$ can be extended locally to all of $\Sigma$ by L'Hospital in this case. We collect the findings in the following proposition (fixing $q=0$ for simplicity).
\begin{proposition}\proplab{VI3}
Consider a PWS system \eqref{zZpm} in a sufficiently small neighborhood of the origin. Suppose furthermore that 
 \begin{align}\eqlab{finaltwofoldcond}
  \begin{cases}
  X_+ (0,\lambda_0 )&>0,\\ 
  Y_+(0,\lambda_0)&=0, \\
  Y_+'(0,\lambda_0)&> 0,
  \end{cases}\quad  \begin{cases}
  X_- (0,\lambda_0 )&< 0,\\ 
  Y_-(0,\lambda_0)&=0, \\
  Y_-'(0,\lambda_0)&<0.
  \end{cases}
  \end{align}
  Then the following holds about system \eqref{zZpm} for $\lambda=\lambda_0$:
  \begin{itemize}
  \item[(i)] The origin is a visible-invisible two-fold.
  \item[(ii)] $\Sigma=\overline{\Sigma_{sl}(\lambda_0)}$ with stable sliding for $x<0$ and unstable sliding for $x>0$.
  \item[(iii)] $X_{sl}(x,\lambda_0)$ is well-defined for all $x\in \Sigma$. 
  \item[(iv)] $(Y_+'-Y_-')(0,0,\lambda_0)>0$.
  \end{itemize}
\end{proposition}
Henceforth, we suppose that \eqref{finaltwofoldcond} holds and that $X_{sl}(x,\lambda_0)>0$ for all $x\in \Sigma$, so that the flow of $X_{sl}$ takes points from stable sliding to unstable sliding. 
These conditions -- which following \propref{VI3} item (iv) and \eqref{Xsl2} imply that
\begin{align}
\text{det}\,Z'>0\eqlab{detZp2}
\end{align} 
 -- correspond to the specific visible-invisible two-fold called $VI_{3}$ in \cite{Kuznetsov2003}. See an illustration of this case in \figref{twofoldVI3}.

\begin{figure}[h!]
\begin{center}
{\includegraphics[width=.573\textwidth]{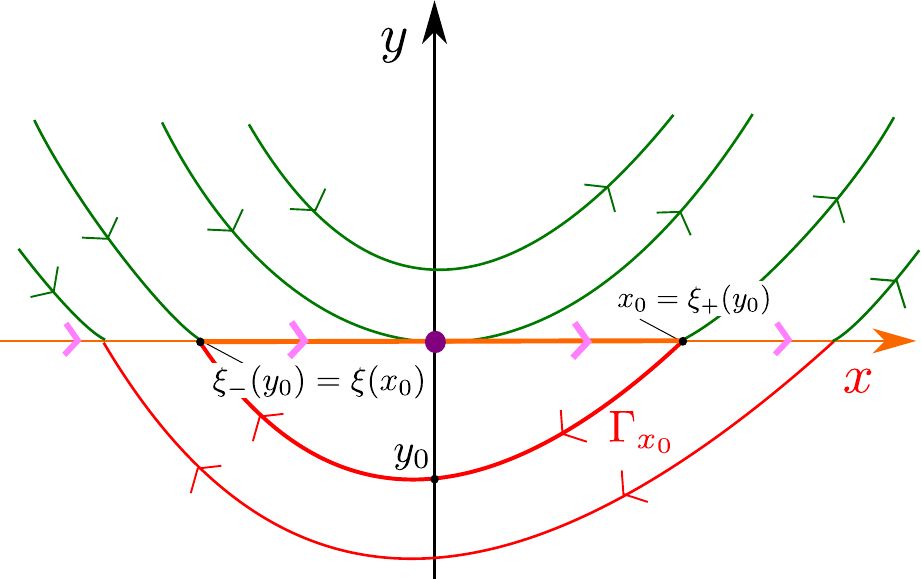}}
% \subfigure[]{\includegraphics[width=.3\textwidth]{./twofoldIV.pdf}}
% \subfigure[]{\includegraphics[width=.3\textwidth]{./twofoldII.pdf}}
% \subfigure[]{\includegraphics[width=.495\textwidth]{layerReducedBlowup.pdf}}
% \input{posusts}
\end{center}
\caption{The $VI_3$ visible-invisible two-fold where $X_{sl}(x,\lambda_0)>0$ (in magenta) for all $x$ locally so that the stable sliding region $x<0$ is connected to the unstable sliding region by the flow of $X_{sl}$ (extended through $x=0$). $\xi$ and $\xi_\pm$ are used in relation to the slow divergence-integral in \secref{slowdiv}.}
\figlab{twofoldVI3}.
\end{figure}
We collect these assumptions and \eqref{versal} into the following hypothesis. 
\begin{enumerate}[resume*]
%  \item 
% \end{enumerate}
    \item \label{assD} Suppose that \eqref{versal} and \eqref{finaltwofoldcond} both hold and that there are $\mu_-<0$ and $\mu_+>0$ such that the PWS system \eqref{zZpm} for $\lambda=\lambda_0$ has stable sliding for all $x\in [\mu_-,0[$ and unstable sliding for $x\in ]0,\mu_+]$ and that $X_{sl}(x,\lambda_0)>0$ for all $x\in [\mu_-,\mu_+]$. Moreover, we assume that $\xi(x)\in [\mu_-,0[$, for each $x\in ]0,\mu_+]$, where $\xi(x)$ is the $x$-value of the first intersection with the $x$-axis of the forward flow of $(x,0)$ following $Z_-$ for $\lambda=\lambda_0$. 
%     \end{itemize}
\end{enumerate}
\subsection{Canards of \eqref{zZ}}\seclab{section-breaking}
In \cite{bonet-reves2018a,kristiansen2015a}, it was independently shown that under the assumption \ref{assD}, the two invariant manifolds for $x<0$ and $x>0$ (which are slow manifolds within the scaling regime defined by $y=\epsilon^2 y_2$, recall the proof of \thmref{slidinghmtsd}) intersect along some $\lambda\sim \lambda_0$ for all $0<\epsilon\ll 1$. Such orbits are also called canards. The reference \cite{kristiansen2015a} used the blowup method, which will also form the basis of our analysis. %In the following, we briefly summarize this approach. 

\section{The slow divergence-integral and canard limit cycles}\seclab{slowdiv}
Consider \eqref{zZ} satisfing \ref{assA}-\ref{assD}. For $\lambda=\lambda_0$ the singular limit (Filippov) system is shown in \figref{twofoldVI3}. The situation is clearly reminiscent of the classical canard situation. In particular, at the level $\lambda=\lambda_0$, we denote by $\Gamma_x$ for $x\in ]0,\mu_+]$, the limit periodic set consisting of the segment $[\xi(x),x]\subset \Sigma$ and the regular orbit of $Z_-$ connecting $(x,0)$ and $(\xi(x),0)$. We call $\Gamma_x$ a canard cycle. We then define the associated slow divergence-integral along the segment $[\xi(x),x]$:
\begin{align}\eqlab{slowdiv}
 I(x) =\int_{\xi(x)}^x \frac{(Y_+-Y_-)^2}{\text{det}\,Z}(u,0,\lambda_0) \phi'\left(\phi^{-1} \left(\frac{-Y_-}{Y_+-Y_-}(u,0,\lambda_0)\right)\right)du,
\end{align}
for $x\in ]0,\mu_+]$. The slow divergence-integral is the integral of the divergence of the vector field \eqref{divergence}, for $\epsilon=0$, computed along the critical manifold w.r.t. the slow time $\tau$ \response{defined by $d\tau=\frac{dx}{X_{sl}(x,\lambda_0)}$}.  It follows from \ref{assD} that $I$ in \eqref{slowdiv} is well-defined.
\smallskip

The following result plays a crucial role in proving \thmref{mainthm}.

\begin{theorem}\thmlab{mainslowdiv}
Let the regularized system \eqref{zZ} satisfy \ref{assA}-\ref{assD}. Suppose that $I(x)$ has exactly $k-1$ simple zeros $x_1<\dots<x_{k-1}$ in $]0,\mu_+[$. If $x_{k}\in ]x_{k-1},\mu_+]$, then there is a smooth function $\lambda=\lambda_c(\epsilon)$, with $\lambda_c(0)=\lambda_0$, such that $Z(z,\phi(y\epsilon^{-2}),\lambda_c(\epsilon))$ has $k$ periodic orbits $\mathcal{O}_1^\epsilon,\dots\mathcal{O}_{k}^\epsilon$, for each $\epsilon\sim 0$ and $\epsilon>0$. The periodic orbit $\mathcal{O}_i^\epsilon$ is isolated, hyperbolic and Hausdorff close to the canard cycle $\Gamma_{x_i}$, for each $i=1,\dots k$.
\end{theorem}

A result similar to \thmref{mainslowdiv} for smooth planar slow-fast systems can be found in \cite{SDICLE1,dumortier2011a}.

Notice that the statement of  \thmref{mainslowdiv} deals only with limit cycles of size $\mathcal{O}(1)$ in the $(x,y)$-phase space. Once the positive simple zeros of the slow divergence integral $I$ are detected, the related canard cycles $\Gamma_{x_1},\dots$ (and hence the limit cycles $\mathcal{O}_1^\epsilon,\dots$) are of size $\mathcal{O}(1)$. Thus, the limit cycles born from the origin $(x,y)=(0,0)$ are not covered by \thmref{mainslowdiv}.

We divide the proof of \thmref{mainslowdiv} into three parts. In the first part we  consider the extended fast-time system $(z',\epsilon')=(\epsilon^2 Z,0)$ and then gain smoothness by applying a cylindrical blow-up (see \secref{cylblowup}). Using the cylindrical blow-up we replace the discontinuity line $\Sigma$ of the PWS system \eqref{zZpm} with a half-cylinder and we show that near the canard trajectories on the top of the cylinder, we are in the framework of \cite{DM-entryexit}. In \cite{DM-entryexit} a very general smooth planar slow-fast model has been studied containing a normally attracting branch of singularities, a normally repelling branch of singularities and a turning point between them (an additional critical curve passing through the turning point is possible). One usually uses the results of \cite{DM-entryexit} for specific slow-fast families by checking the assumptions in \cite{DM-entryexit} (see for example \cite{DDRR,HuzakPrey,Zhu}). We do the same here. In the second part (\secref{diffmap}) we find the structure of the difference map of \eqref{zZ} near $\Gamma_{x}$ using \cite{DM-entryexit} and Proposition \ref{prop-passage-hyperbolic} (Appendix \ref{appendix}) near the hyperbolic edge of the cylinder. In the third  part (\secref{conclusions}) we establish a \response{one-to-one} correspondence between simple zeros of the slow divergence-integral \eqref{slowdiv} and simple zeros of the difference map by choosing a suitable control function $\lambda=\lambda_c(\epsilon)$, following \cite{Benoit,DM-entryexit}.

\subsection{Cylindrical blow-up}\seclab{cylblowup} First we introduce the following scaling:
\[\lambda=\lambda_0+\epsilon\widetilde{\lambda}\]
where $\widetilde{\lambda}\sim 0$ is called a regular breaking parameter. 
We study the system $Z$ given in \eqref{zZ} in nonsmooth limit $\varepsilon\rightarrow 0$ in the classical way, see e.g. \cite{kristiansen2018a}. We consider the extended fast-time system $(z',\epsilon')=(\epsilon^2 Z,0)$ and apply the cylindrical blow-up
\begin{align}\eqlab{blowup1}
 (r,(\bar y,\bar \epsilon))\mapsto \begin{cases}
                                    y&=r^2\bar y,\\
                                    \epsilon &=r\bar \epsilon,
                                   \end{cases}
\end{align}
with $r\ge 0$, $(\bar y,\bar \epsilon)\in \mathbb{S}^1$ and $\bar\epsilon\ge 0$. Let $\overline F$ denote the vector field on $(x,r,(\bar y,\bar \epsilon))$, i.e. the pullback of $(\epsilon^2 Z,0)$ under \eqref{blowup1}. We then perform desingularization by division of the right hand side by $\bar \epsilon^2$. In other words, it is $\widehat F:=\bar \epsilon^{-2} \overline F$ that we shall study. To study the dynamics of $\widehat F$ in a neighborhood of the cylinder, we use different charts. Based upon Section \ref{section-primary-fam} and Section \ref{section-primary-phase}, we illustrate the transformation and the properties of $\widehat F$ in \figref{blowup}(a). %(Notice in particular how the edge of the cylinder, corresponding to $r=0, \bar y=\pm 1$, is hyperbolic for $x\ne 0$. This is a consequence of the desingularization leading to $\widehat F$.) 
\begin{figure}[h!]
 	\begin{center}
 		%\subfigure[]{\includegraphics[width=.465\textwidth]{./figures/hesterPPlaneNew.pdf}}
 		%\subfigure[]{\includegraphics[width=.49\textwidth]{./figures/hesterMatlabNew.pdf}}
 		\subfigure[]{\includegraphics[width=.9\textwidth]{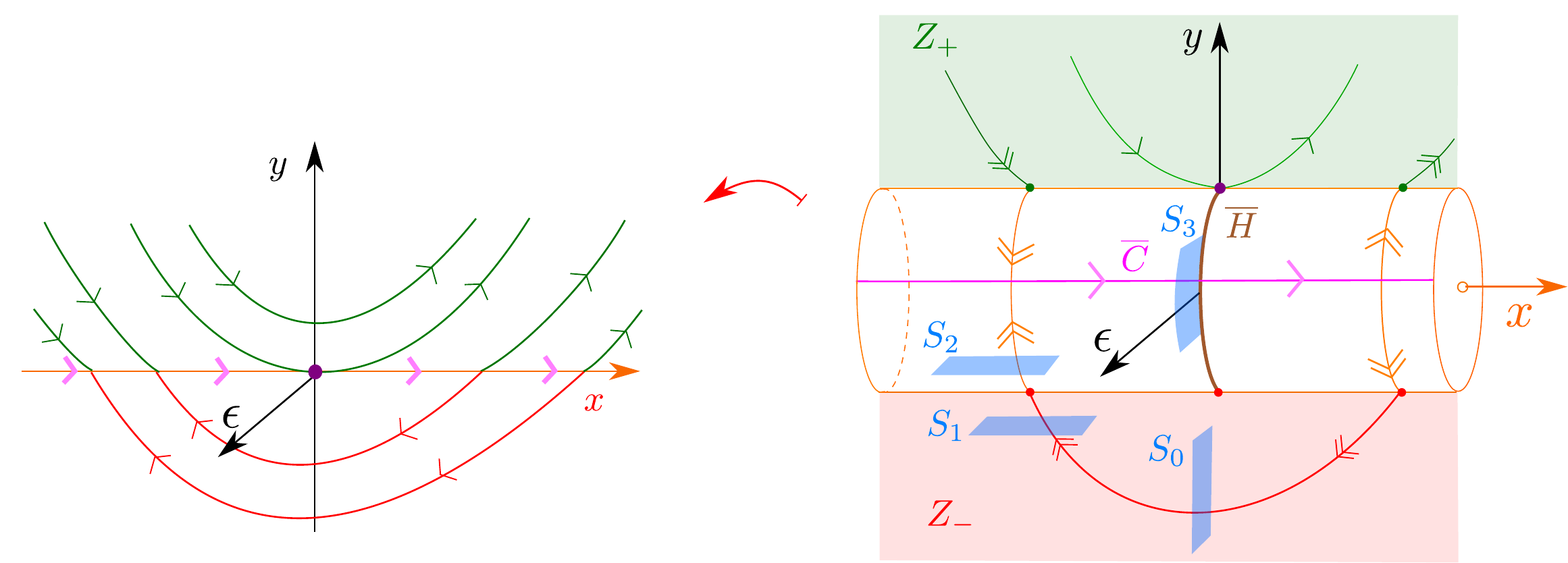}}
 		\subfigure[]{\includegraphics[width=.905\textwidth]{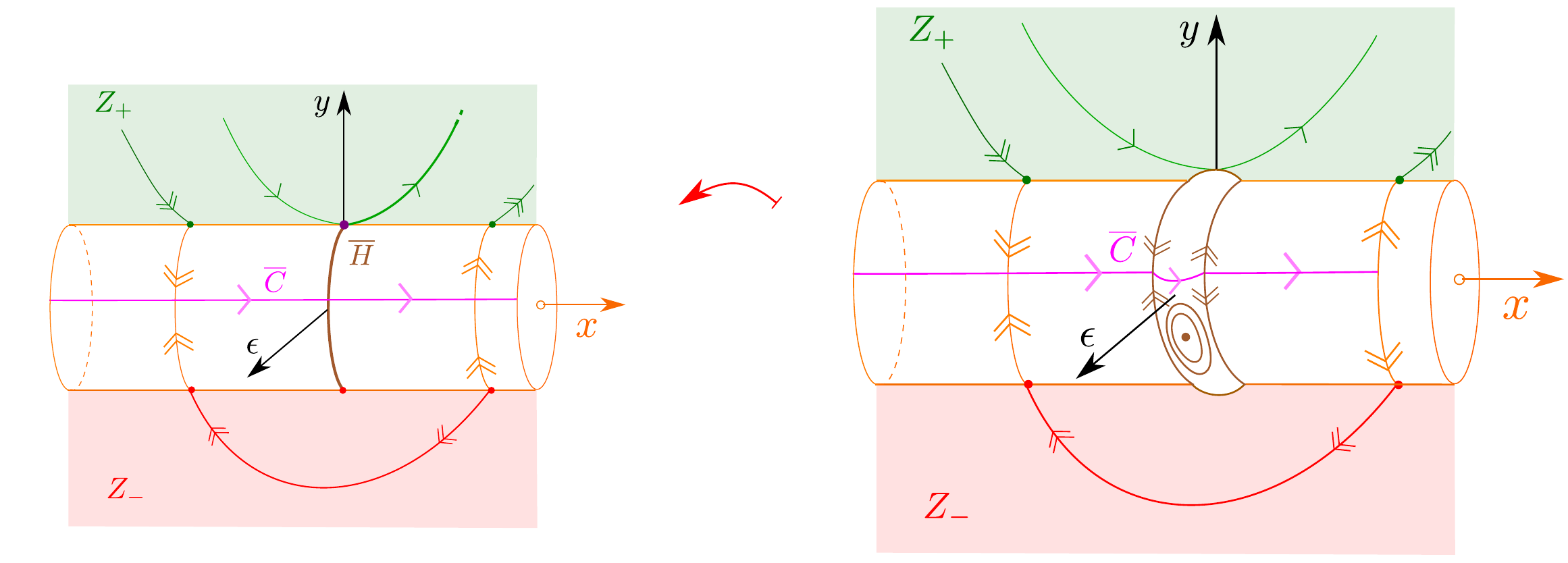}}
 		% \subfigure[]{\includegraphics[width=.49\textwidth]{./mu.pdf}}
 		\caption{The two consecutive blowup transformations. In (a): Under the assumptions \ref{assA}-\ref{assC} we gain smoothness of $y=0$ at $\epsilon=0$ through a cylindrical blowup transformation. On the blow-up cylinder we find two critical sets $\overline C$ and $\overline H$, the former being normally hyperbolic away from the intersection with $\overline H$. The section $S_i$, $i=0,\ldots,3$ are used in the proof of Proposition \ref{prop-diffmap}. In (b): We blowup $\overline H$ through another cylindrical blowup transformation. In this way, we gain hyperbolicity of $\overline C$. Hyperbolic directions are indicated by double-headed arrows whereas single headed arrows are slow or nonhyperbolic directions. }\figlab{blowup}
 	\end{center}
 	%  \end{figure}
 	% \caption{$q=r_2q_2$, $x=r_2x_d$, $\epsilon=r_2^2$}
 \end{figure}

\subsubsection{Dynamics in the scaling chart $\bar \epsilon=1$}\label{section-primary-fam}

We consider the chart-specific coordinate $y_2$ defined by $y=\epsilon^2 y_2$, with $(x,y_2)$ kept in a large compact subset of $\mathbb{R}^2$, $\epsilon> 0$ and $\epsilon\sim 0$. Inserting this into the extended system $(z',\epsilon')=(\epsilon^2 Z,0)$ produces the following equations: %(after division by $\bar\epsilon^2=1$):
\begin{equation}\eqlab{global}
\begin{aligned}
 \dot x &=\epsilon^2\left(X_+(x,\epsilon^2y_2,\lambda_0+\epsilon\widetilde{\lambda})\phi(y_2)+X_-(x,\epsilon^2y_2,\lambda_0+\epsilon\widetilde{\lambda})(1-\phi(y_2))\right),\\
 \dot y_2 &=Y_+(x,\epsilon^2y_2,\lambda_0+\epsilon\widetilde{\lambda})\phi(y_2)+Y_-(x,\epsilon^2y_2,\lambda_0+\epsilon\widetilde{\lambda})(1-\phi(y_2)).
\end{aligned}
 \end{equation}
When $\epsilon=0$, system \eqref{global} becomes 
\begin{equation}\eqlab{globaleps=0}
\begin{aligned}
 \dot x &=0,\\
 \dot y_2 &=Y_+(x,0,\lambda_0)\phi(y_2)+Y_-(x,0,\lambda_0)(1-\phi(y_2)).
\end{aligned}
 \end{equation}
 The critical set of \eqref{globaleps=0} is given by the union of two critical manifolds:
  \begin{equation}\label{manifold-H}
  \overline H:=\{(0,y_2):y_2\in\mathbb{R}\}
 \nonumber
 \end{equation}
and the curve $\overline C$ given by
 \begin{equation}\label{manifold-C2}
 y_2=\phi^{-1}\left(\frac{-Y_-}{Y_+-Y_-}(x,0,\lambda_0)\right).\nonumber
 \end{equation}
 Notice that at the point 
  \[p_0=\left(0, \phi^{-1}\left(\frac{-Y_-'}{Y_+'-Y_-'}(0,0,\lambda_0)\right)\right)\]
 an intersection of $\overline H$ and $\overline C$ appears. All the singularities on $\overline H$ are nilpotent
except for $p_0$ which is linearly zero. In the rest of this section we show that the slow-fast system \eqref{global} satisfies Assumptions T0--T6 in \cite{DM-entryexit} along the critical curve $\overline{C}$. Then we can use \cite[Theorem 4]{DM-entryexit} and prove Proposition \ref{prop-diffmap} in \secref{diffmap}. Theorem 4 says that the leading term of the integral of divergence of the vector field \eqref{global}, computed along canard orbits near $\overline{C}$ between $\xi(x)<0$ and $x>0$, is $\frac{I(x)}{\epsilon^2}$ with $I$ defined in \eqref{slowdiv} (see for example the exponent of the exponential term in (\ref{equ-assumption})). This term remains dominant in the expression for the difference map of \eqref{zZ} near $\Gamma_{x}$, see \eqref{diff-form-pm}.

 The singularities on $\overline C$ are normally attracting when $x<0$, normally repelling when $x>0$ and the slow dynamics $X_{sl}(x,\lambda_0)$ -- given in \eqref{Xsl2} -- is regular, pointing from the attracting part to the repelling part of $\overline{C}$. Thus, if we denote the vector field in \eqref{global} by $\widehat{F}_S$ and if $M_x$ is any local $C^n$ center manifold of $\widehat{F}_S^1:=\widehat{F}_S+0\frac{\partial}{\partial\epsilon}$ at normally hyperbolic singularity $x\in \overline{C}$, then $\frac{1}{\epsilon^2}\widehat{F}_S^1|_{M_x}$ is a local flow box containing $\overline C$ and pointing from the left to the right. (The exponent $2$ in the term $\epsilon^2$ is often called the order of degeneracy.) This implies that Assumptions T0--T2 of \cite{DM-entryexit} are satisfied. 
 It remains to show that \eqref{global} satisfies Assumptions T3--T6 of \cite{DM-entryexit} in an $(\epsilon,\widetilde\lambda)$-uniform neighborhood of the turning point $p_0$. In order to do that, we have to blow up the degenerate line $\overline H=\{x=0\}$, inside the slow-fast system \eqref{global}, to a half-cylinder (see \figref{blowup}(b)). For the sake of readability of \secref{slowdiv}, we prove that Assumptions T3--T6 are satisfied in Appendix \ref{appendix-0}.
 
 \smallskip

\subsubsection{Dynamics in the phase directional charts $\bar y=\pm 1$}\label{section-primary-phase} We keep $x\in[\mu_-,\mu_+]$ uniformly away from $x=0$. In the chart $\bar y=-1$ associated with \eqref{blowup1} and the chart-specific coordinates $(r_1,\epsilon_1)$ such that $(y,\epsilon)=(-r_1^2,r_1\epsilon_1)$ the extended system $(z',\epsilon')=(\epsilon^2 Z,0)$ becomes (after division by $\epsilon_1^2>0$):
\begin{equation}\eqlab{PDChart-1}
\begin{aligned}
    \dot x &= r_1^2 \left(X_+(x,-r_1^2,\lambda_0+r_1\epsilon_1\widetilde{\lambda})\phi_-(\epsilon_1^2)+X_-(x,-r_1^2,\lambda_0+r_1\epsilon_1\widetilde{\lambda})(1-\phi_-(\epsilon_1^2))\right),\\
    \dot r_1 &=-\frac{1}{2}r_1 \left(Y_+(x,-r_1^2,\lambda_0+r_1\epsilon_1\widetilde{\lambda})\phi_-(\epsilon_1^2)+Y_-(x,-r_1^2,\lambda_0+r_1\epsilon_1\widetilde{\lambda})(1-\phi_-(\epsilon_1^2))\right),\\
    \dot \epsilon_1 &=\frac{1}{2}\epsilon_1 \left(Y_+(x,-r_1^2,\lambda_0+r_1\epsilon_1\widetilde{\lambda})\phi_-(\epsilon_1^2)+Y_-(x,-r_1^2,\lambda_0+r_1\epsilon_1\widetilde{\lambda})(1-\phi_-(\epsilon_1^2))\right)
\end{aligned}
\end{equation}
where $\phi_-$ is defined in \ref{assC}. The edge of the cylinder, corresponding to $r_1=\epsilon_1=0$, consists of semi-hyperbolic singularities of \eqref{PDChart-1}. The eigenvalues of the linearization at $(x,0,0)$ are given by $(0,-\frac{Y_-(x,0,\lambda_0)}{2},\frac{Y_-(x,0,\lambda_0)}{2})$. Let's recall that $Y_-(x,0,\lambda_0)<0$ when $x>0$ and $Y_-(x,0,\lambda_0)>0$ when $x<0$. The form of the transition map near the edge of the cylinder, with $x<0$ (resp. $x>0$), by following the orbits of \eqref{PDChart-1} in forward (resp. backward) time is given in Proposition \ref{prop-passage-hyperbolic} in Appendix \ref{appendix}.

Although the phase directional chart $\bar{y}=1$ is not relevant to the present study, we include it here for sake of completeness. Writing $(y,\epsilon)=(r_2^2,r_2\epsilon_2)$, the  extended system changes (after division by $\epsilon_2^2>0$) into 
\begin{equation}\eqlab{PDChart+1}
\begin{aligned}
    \dot x &= r_2^2 \left(X_+(x,r_2^2,\lambda_0+r_1\epsilon_1\widetilde{\lambda})\phi_+(\epsilon_2^2)+X_-(x,r_2^2,\lambda_0+r_1\epsilon_1\widetilde{\lambda})(1-\phi_+(\epsilon_2^2))\right),\\
    \dot r_2 &=\frac{1}{2}r_2 \left(Y_+(x,r_2^2,\lambda_0+r_1\epsilon_1\widetilde{\lambda})\phi_+(\epsilon_2^2)+Y_-(x,r_2^2,\lambda_0+r_1\epsilon_1\widetilde{\lambda})(1-\phi_+(\epsilon_2^2))\right),\\
    \dot \epsilon_2 &=-\frac{1}{2}\epsilon_2 \left(Y_+(x,r_2^2,\lambda_0+r_1\epsilon_1\widetilde{\lambda})\phi_+(\epsilon_2^2)+Y_-(x,r_2^2,\lambda_0+r_1\epsilon_1\widetilde{\lambda})(1-\phi_+(\epsilon_2^2))\right),
\end{aligned}
\end{equation}
with $\phi_+$ introduced in \ref{assC}. The study of \eqref{PDChart+1} near $r_2=\epsilon_2=0$ is similar to the study of \eqref{PDChart-1} near $r_1=\epsilon_1=0$. The points $(x,0,0)$, for $x\ne 0$, are semi-hyperbolic singularities of \eqref{PDChart+1} with eigenvalues $(0,\frac{Y_+(x,0,\lambda_0)}{2},-\frac{Y_+(x,0,\lambda_0)}{2})$.

\subsection{The difference map}\seclab{diffmap} Denote by $\xi_-(y)<0$ (resp. $\xi_+(y)>0$), with $y<0$, the $x$-value of the intersection with the $x$-axis  of the forward (resp. backward) flow of $(0,y)$ following $Z_-$, for $\lambda=\lambda_0$. Let $\mu_1<\mu_2<0$ be arbitrary and fixed real numbers such that $\xi_+([\mu_1,\mu_2])\subset ]0,\mu_+[$ (and hence $\xi_-([\mu_1,\mu_2])\subset ]\mu_-,0[$ by \ref{assD}). We define a section $S_0\subset\{x=0\}$ parametrized by $y\in [\mu_1,\mu_2]$ and $\epsilon\in [0,\epsilon_0]$ where $\epsilon_0$ is a small positive constant. We also define a section $S_3\subset\{x=0\}$, parametrized by $y_2\sim \phi^{-1}\left(\frac{-Y_-'}{Y_+'-Y_-'}(0,0,\lambda_0)\right)$ and $\epsilon\in [0,\epsilon_0]$, where the coordinate $y_2$ is introduced in Section \ref{section-primary-fam}. We denote by $\Delta_{-}$ (resp. $\Delta_+$) the transition map between $S_0$ and $S_3$ following the trajectories of the blown-up vector field $\widehat{F}$ in forward (resp. backward) time. It is clear that the zeros of the difference map
\begin{equation}\label{difference-map}
y_2=\Delta(y,\epsilon,\widetilde\lambda):=\Delta_-(y,\epsilon,\widetilde\lambda)-\Delta_+(y,\epsilon,\widetilde\lambda),
\end{equation}
with $\epsilon>0$,
correspond to periodic orbits of \eqref{zZ}.

\begin{proposition}\label{prop-diffmap}
The transition maps $\Delta_{\pm}$ have the following form:
\begin{equation}\eqlab{diff-form-pm}
 \Delta_\pm(y,\epsilon,\widetilde\lambda)=f_{\pm}(\epsilon,\widetilde\lambda)-\exp\frac{1}{\epsilon^2}\left(I_\pm(y)+o_\pm(1)\right), \ y\in[\mu_1,\mu_2],\ (\epsilon,\widetilde\lambda)\sim(0,0),
\end{equation}
where $f_\pm$ are smooth functions, $(f_--f_+)(0,0)=0$, $\frac{\partial(f_--f_+)}{\partial\widetilde\lambda}(0,0)\ne 0$, $o_\pm(1)$  tend to zero as $\epsilon\to 0$, uniformly in $(y,\widetilde\lambda)$, and where 
 \begin{equation}\label{SDI}
 \begin{aligned}
     I_\pm(y)=&\int_{\xi_\pm(y)}^0\frac{(Y_+-Y_-)^2}{\textnormal{det}\,Z}(x,0,\lambda_0) \phi'\left(\phi^{-1} \left(\frac{-Y_-}{Y_+-Y_-}(x,0,\lambda_0)\right)\right)dx<0.
 \end{aligned}
 \end{equation}
\end{proposition}
\begin{proof}
We treat the forward transition map $\Delta_-$ (the backward transition map $\Delta_+$ can be studied in similar fashion). We split up the forward transition map $\Delta_-$ between $S_0$ and $S_3$ in three parts (see sketch of sections in \figref{blowup}(a)):\\
(a) We define a section $S_1\subset \{r_1=r_{10}\}$ parametrized by $x\in J\subset [\mu_-,0[$, $J$ being a segment, and $\epsilon_1\in[0,\frac{\epsilon_0}{r_{10}}]$ where $r_{10}>0$ is a small constant and $(x,r_1,\epsilon_1)$ are the coordinates of \eqref{PDChart-1}. The segment $J$ is chosen large enough such that the transition map $x=\Delta_{01}(y,\epsilon,\widetilde\lambda)$ between $S_0$ and $S_1$ is well defined. Notice that $\epsilon_1=\frac{\epsilon}{r_{10}}$. Since $Z_-$ has no singularities between $S_0$ and $S_1$ and the passage between $S_0$ and $S_1$ is located outside a fixed neighborhood of $y=0$, it is clear that $\Delta_{01}$ is smooth in $(y,\epsilon,\widetilde\lambda)$ (see also Lemma \ref{lemma-regular}).\\
(b) Define a section $S_2\subset \{\epsilon_1=\epsilon_{10}\}$ parametrized by $\bar x\in\bar J\subset [\mu_-,0[$, $\bar J$ being a segment and $\bar r_1\in [0,\frac{\epsilon_0}{{\epsilon}_{10}}]$, with a small positive constant $\epsilon_{10}$. Following Proposition \ref{prop-passage-hyperbolic}, the transition map $\bar x=\Delta_{12}(x,\epsilon_1,\widetilde\lambda)$  between $S_1$ and $S_2$ w.r.t. \eqref{PDChart-1} can be written as 
\[\Delta_{12}(x,\frac{\epsilon}{r_{10}},\widetilde\lambda)=g_{12}(x,\widetilde\lambda)+O(\epsilon\log\epsilon^{-1}), \ \epsilon \to 0.\]
Notice that $\bar{r}_1=\frac{r_{10}\epsilon_1}{\epsilon_{10}}=\frac{\epsilon}{\epsilon_{10}}$.\\
(c) The transition map $y_2=\Delta_{23}(\bar{x},\epsilon,\widetilde\lambda)$ between $S_2$ and $S_3$ following the trajectories of the smooth slow-fast system \eqref{global} has the following form (see \cite[Theorem 4]{DM-entryexit}):
\begin{equation}\label{equ-assumption}
\Delta_{23}(\bar{x},\epsilon,\widetilde\lambda)=f_-(\epsilon,\widetilde\lambda)-\exp\frac{1}{\epsilon^2}\left(\bar I(\bar{x})+\kappa_1(\bar{x},\epsilon,\widetilde\lambda)+\kappa_2(\epsilon,\widetilde\lambda)\epsilon^2\log\epsilon\right)
\end{equation}
where $f_-$, $\kappa_1$ and $\kappa_2$ are smooth, including at $\epsilon=0$, $\kappa_1=O(\epsilon)$ and $\bar{I}(\bar{x})<0$ is the slow divergence-integral of the form (\ref{SDI}) computed between $\bar{x}$ and $0$. We have the negative sign in front of the exponential term due to the chosen parametrization of $S_2$ and $S_3$.

Combining (a), (b) and (c), we obtain \eqref{diff-form-pm}. We use that $g_{12}(\Delta_{01}(y,0,0),0)=\xi_-(y)$. Since Assumption T6 of \cite{DM-entryexit} is satisfied (see Section \ref{appendix-0}), the function $f_--f_+$ has the property given in Proposition \ref{prop-diffmap}, where $f_+$ is obtained in a similar way by studying the backward transition map $\Delta_+$. 
\end{proof}

\subsection{Conclusions}\seclab{conclusions} Suppose that $I(x)$, defined in \eqref{slowdiv}, has exactly $k-1$ simple zeros $x_1<\dots<x_{k-1}$ in $]0,\mu_+[$. Let the segment $[\mu_1,\mu_2]$ from \secref{diffmap} be large enough such that $x_1,\dots,x_{k-1}\in\xi_+([\mu_1,\mu_2])$.
Using the property of $f_--f_+$ given in Proposition \ref{prop-diffmap} ($\widetilde\lambda$ is the breaking parameter) and the implicit function theorem, we find a smooth function $\widetilde\lambda=\widetilde\lambda_c(\epsilon)$, with $\widetilde\lambda_c(0)=0$, such that $(f_--f_+)(\epsilon,\widetilde\lambda_c(\epsilon))=0$ for all small $\epsilon\ge 0$. Now, the difference map $\Delta$, given in (\ref{difference-map}), can be written as
\[\Delta(y,\epsilon,\widetilde\lambda_c(\epsilon))=\exp\frac{1}{\epsilon^2}\left(I_+(y)+o_+(1)\right)-\exp\frac{1}{\epsilon^2}\left(I_-(y)+o_-(1)\right)\]
for new functions $o_\pm(1)$ tending to zero as $\epsilon\to 0$, uniformly in $y$. This implies that the zeros of $\Delta(y,\epsilon,\widetilde\lambda_c(\epsilon))$ w.r.t. $y$ are solutions of the equation
\begin{equation}\label{equation-SDI-zeros}
 I_-(y)-I_+(y)+o(1)=0,   
\end{equation}
where $o(1)\to 0$ when $\epsilon\to 0$ (uniformly in $y$). Notice that $\xi(\xi_+(y))=\xi_-(y)$, and therefore $I_-(y)-I_+(y)=I(\xi_+(y))$. We conclude that $y_1,\dots,y_{k-1}$, defined by $\xi_+(y_i)=x_i$, are simple zeros of $I_--I_+$ ($\xi_+$ is a diffeomorphism). Using the implicit function theorem once more, we find that (\ref{equation-SDI-zeros}) has $k-1$ simple solutions for each small $\epsilon>0$, perturbing from $y_1,\dots,y_{k-1}$. They correspond to hyperbolic canard limit cycles of $Z(z,\phi(y\epsilon^{-2}),\lambda_0+\epsilon\widetilde\lambda_c(\epsilon))$ close to $\Gamma_{x_1},\dots,\Gamma_{x_{k-1}}$. It is not difficult to see that using the control function $\widetilde\lambda_c(\epsilon)$ we can construct one extra hyperbolic limit cycle, Hausdorff close to $\Gamma_{x_k}$, with $x_k\in ]x_{k-1},\mu_+[$, surrounding the $k-1$ limit cycles (see \cite{SDICLE1,dumortier2011a}). This completes the proof of \thmref{mainslowdiv}.

\begin{remark}
Notice that the parameter $\lambda$ in our model \eqref{zZ} is one-dimensional and we don't need additional parameters in the statement of \thmref{mainslowdiv} to prove \thmref{mainthm}. Of course \thmref{mainslowdiv} remains true if $Z_\pm$ depend smoothly on  finite-dimensional extra parameter. 
\end{remark}

\begin{remark}\label{remark-Prop-important}
Suppose that the slow divergence-integral $I$ has a simple zero at $x=x_0\in ]0,\mu_+]$. Then for each small $\epsilon>0$, the $\lambda$-family  $Z(z,\phi(y\epsilon^{-2}),\lambda)$ undergoes a saddle-node bifurcation of limit cycles near $\Gamma_{x_0}$ as we vary $\lambda\sim\lambda_0$. Notice that the parameter $\lambda$ in this result--as opposed to \thmref{mainslowdiv} with unbroken $\lambda$-- becomes broken. If the slow divergence-integral $I$ has a zero of multiplicity $l\ge 1$ at $x=x_0$, then $Z(z,\phi(y\epsilon^{-2}),\lambda)$ can have at most $l+1$ limit cycles (counting multiplicity) Hausdorff close to $\Gamma_{x_0}$ for each small $\epsilon>0$ and $\lambda\sim \lambda_0$, and, if $I(x_0)<0$ (resp. $I(x_0)>0$), then at most one limit cycle can be born from $\Gamma_{x_0}$. The limit cycle, if it exists, is hyperbolic and attracting (resp. repelling). These results can be proved by using Proposition \ref{prop-diffmap}. The proof is similar to the proof of \cite[Theorem 4.3]{dumortier2011a}.  \end{remark}

\section{Proof of \thmref{mainthm}}\seclab{final}

\response{To prove \thmref{mainthm} we now use} \thmref{mainslowdiv}. We consider $Z_\pm(\cdot,\lambda)$ and suppose that \ref{assD} holds with $\lambda_0=0$. Moreover, we will suppose that $Z_-$ is invariant under the symmetry $(x,t)\mapsto (-x,-t)$ for $\lambda=0$: 

\begin{enumerate}[resume*]
    \item \label{assE} Let $\Gamma(x,y)=(-x,y)$ then we assume $D\Gamma^{-1} (Z_- \circ \Gamma) = -Z_-$ for $\lambda=0$. 
\end{enumerate}

Based upon the following simple result, this leads to a significant simplification of the calculations that follow.  
\begin{lemma}
Assume that {assumption \ref{assE}} holds. Then $\xi(x)=-x$, recall \eqref{slowdiv}, and $I(x)$ has a smooth extension onto a neighborhood of $x=0$ which is an odd function in $x$. 
% We consider 
\end{lemma}
\begin{proof}
From the symmetry, we have that if $(\xi_-(y),0)$ is the first intersection with $\Sigma$ by the forward flow of $(0,y)$ then $(\xi_+(y),0)$ with $\xi_+(y)=-\xi_-(y)$ is the first intersection with $\Sigma$ by the backward flow. 
\end{proof}
In the following, while we continue to use $()'$ to denote the partial derivative with respect to $x$ evaluated at $(x,y,\lambda)=(0,0,0)$, we will also use $()''$ to indicate the second order partial derivative with respect to $x$ also evaluated at $(x,y,\lambda)=(0,0,0)$.

We then proceed to Taylor expand $I(x)$ around $x=0$. Let $y_{2c} = \phi^{-1}\left(\frac{-Y_-'}{Y_+'-Y_-'}\right)$ and recall that $$ \text{det}\,Z(x,0,0) = (X_-Y_+-X_+Y_-)(x,0,0).$$  

Since $Z_-$ is assumed to be $\Gamma$-symmetric, see \ref{assE}, we have that $x\mapsto X_-(x,0,0)$ is even whereas $x\mapsto Y_-(x,0,0)$ is odd. Consequently, $X_-'=Y_-''=0$ and 
\begin{align*}
 \text{det}\,Z'' = X_-Y_+''-2X_+'Y_-'.
\end{align*}
Then from \cite[Eq. 4.13]{bonet-reves2018a} we have that
\begin{align*}
 I(x) &= \frac23 x^3 \bigg(\frac12 (Y_+'-Y_-') \left(\frac{Y_+''}{\text{det}\,Z'}-\frac{(Y_+'-Y_-')\text{det}\,Z''}{2(\text{det}\,Z')^2}\right)\phi'(y_{2c})+\frac{\phi''(y_{2c})}{\phi'(y_{2c})}\frac{Y_+''Y_-'}{2\text{det}\,Z'}\bigg)+\mathcal O(x^5),
\end{align*}
under the assumption \ref{assE}, recall also \ref{assD} and \eqref{detZp2}.

The regularization function satisfies assumptions \ref{assA}-\ref{assC}. In particular, it is invertible and $\phi'>0$, but \ref{assA}-\ref{assC} do not impose further restrictions on the higher order partial derivatives $\phi$ at any point. %It will therefore in the following be convenient to think of these as parameters. 
Suppose:

\begin{enumerate}[resume*]
    \item \label{assF} $Y_+''\ne 0$. 
\end{enumerate}

It then follows (see also \ref{assD} and \eqref{detZp2}) that $I^{(3)}(0)$ can have either sign, depending on $\phi''(y_{2c})$. In fact, seeing that $I^{(3)}(0)$ depends upon $\phi''(y_{2c})$ in an affine way -- with a coefficient of $\phi''(y_{2c})$ that is nonzero -- there is a unique value of $\phi''(y_{2c})$ (for every $\phi'(y_{2c})>0$) for which $I^{(3)}(0)=0$. The following lemma allow us to generalize this result to any odd derivative of $I$ at $x=0$.
\begin{lemma}\lemmalab{I2k1}
 $I^{(2k+1)}(0)$ for $k\in \mathbb N$ depends upon $\phi'(y_{2c}),\ldots,\phi^{(2k)}(y_{2c})$ and takes the following form:
 \begin{align}
  I^{(2k+1)}(0) = J_{2k-1}(\phi'(y_{2c}),\ldots,\phi^{(2k-1)}(y_{2c}))+\frac{1}{\phi'(y_{2c})^{2k-1}} C_{2k} \phi^{(2k)}(y_{2c}),\eqlab{I2k1expr}
 \end{align}
 where \response{$J_{2k-1}:\mathbb R_+\times  \underbrace{\mathbb R\times \cdots \times \mathbb R}_{2k-2\, \textnormal{copies}}  \rightarrow \mathbb R$ is a smooth function and} where
 \begin{align}\eqlab{C2k}
  C_{2k}= \frac{4k(Y_+'-Y_-')^2}{\textnormal{det}Z'} \left(\frac{Y_+''Y_-'}{2(Y_+'-Y_-')^2}\right)^{2k-1}.
 \end{align}
 In particular, $C_{2k}\ne 0$ whenever assumption \ref{assD} and \ref{assF} hold. 
\end{lemma}
\begin{proof}
 For simplicity write
 \begin{align*}
  g(x) = \phi^{-1}\left(\frac{-Y_-}{Y_+-Y_-}(x,0,0)\right),\quad h(x) = \frac{(Y_+-Y_-)^2}{\text{det}\,Z}(x,0,0).
 \end{align*}
 Then the integrand of $I(x)$ is $$i(u):=h(u)\phi'(g(u)).$$ Notice that $g$ and $h$ both have ``$0/0$'' at $x=0$, but each has a smooth extension to $x=0$ due to the assumption of the two-fold by L'Hospital, recall also \propref{VI3}. In particular,
 \begin{align*}
  g(0) = y_{2c}:=\phi^{-1}\left(\frac{-Y_-'}{Y_+'-Y_-'}\right).
 \end{align*}
%with $()'$ the partial derivative with respect to $x$ at $x=0$. 
Moreover,
\begin{align*}
 g'(0) = \frac{1}{\phi'(y_{2c})}\frac{Y_+''Y_-'}{2(Y_+'-Y_-')^2},
\end{align*}
using assumption \ref{assF}. 

In the same way, $h(0)=0$ and
\begin{align*}
 h'(0)=\frac{(Y_+'-Y_-')^2}{\text{det}\,Z'}.
\end{align*}
We compute the partial derivatives of $i(x)$ of even degree using the Fa\'a di Bruno rule:
\begin{align}
 i^{(2k)}(0) = \sum_{m=0}^{2k-1} \begin{pmatrix}
                            2k\\m
                           \end{pmatrix} h^{(2k-m)}(0) \sum_{n=1}^m \phi^{(n+1)}(y_{2c}) B_{m,n}(g'(0),\ldots,g^{(m-n+1)}(0)),\eqlab{i2kexpr}
\end{align}
where $B_{m,n}$ are the Bell polynomials. Here we have used that $h(0)=0$. Each $g^{(l)}(0)$ can be written in terms of $\phi'(y_{2c}),\ldots,\phi^{(l)}(y_{2c})$ (as well as the partial derivatives of $Y_\pm$). This follows from the rule of inverse differentiation. 
%\response{The stated properties of $J_{2k-1}$ follows and to complete the proof we just have to show the explicit expression for the coefficient of $\phi^{(2k)}(y_{2c})$}, including the expression for $C_{2k}$. For this purpose, 
\response{To show  the explicit expression for the coefficient of $\phi^{(2k)}(y_{2c})$, including the expression for $C_{2k}$, we consider the term in \eqref{i2kexpr} with $n=m=2k-1$}:
\begin{align*}
 \begin{pmatrix} 2k \\ 2k-1\end{pmatrix} h'(0) \phi^{(2k)}(y_{2c}) B_{2k-1,2k-1}(g'(0))=2kh'(0) \phi^{(2k)}(y_{2c}) g'(0)^{2k-1}
\end{align*}
using that $B_{n,n}(x)=x^n$. By the Leibniz integral rule, the result -- \response{including the stated properties of $J_{2k-1}$} -- then follows.
%  
%  In particular, upon expanding $I(x)$ in $x$ we obtain a term of order $x^{2k+1}$ of the following form:
%  \begin{align*}
%    &\int_{-x}^x \frac{(Y_+'-Y_-')^2}{\text{det}\,Z'} u \frac{1}{(2k-1)!} \phi^{(2k)}(y_{2c}) \left(\frac{1}{\phi'(y_{2c})}\frac{Y_+''Y_-'}{2(Y_+'-Y_-')^2}\right)^{2k-1} u^{2k-1} du \\
%    &= \frac{4k}{(2k+1)!} \frac{(Y_+'-Y_-')^2}{\text{det}\,Z'} \left(\frac{1}{\phi'(y_{2c})}\frac{Y_+''Y_-'}{2(Y_+'-Y_-')^2}\right)^{2k-1}\phi^{(2k)}(y_{2c}) x^{2k+1}.
%  \end{align*}
% All other terms of this order depend on $\phi^{(j)}$ with $j<2k$. 
\end{proof}
%\response{It is clear that $J_{2k-1}$ is in fact a rational function, but we will not need this.}

We now have the following: If we assume \ref{assF} then for each $k\in \mathbb N$ there is a unique value:
\begin{align}\eqlab{Psi0}
    -C_{2k}^{-1} J_{2k-1}(\phi'(y_{2c}),\ldots,\phi^{(2k-1)}(y_{2c}))\phi'(y_{2c})^{2k-1},
\end{align}
of $\phi^{(2k)}(y_{2c})$ (for fixed values of the derivatives of lower order $\phi'(y_{2c}),\ldots,\phi^{(2k-1)}(y_{2c})$) such that $I^{(2k+1)}(0)=0$. %Call this value $\Phi_0^{(2k)}$.

We can then prove the following result. 
% We therefore define $I_2(x_2,\delta):=\delta^{-2k} x^I(\delta x_2)$ and look for roots of $I_2(\cdots,\delta)$ rather than roots of $I$. Setting $\delta=0$, $I_2(x_2,0)=P_
%By a simple implicit function theorem argument we then obtain the following:
\begin{theorem}\thmlab{thmhere}
 Suppose that \ref{assF} holds. Then for each $k\in \mathbb N$ there is a regularization function $\phi_k$ satisfying \ref{assA}-\ref{assC} so that $I(x)$ has $k-1$ simple positive roots. 
\end{theorem}
\begin{proof}
 For each $k\in \mathbb N$, we first put $\Phi_1^{(2k)}=1$ and define the numbers $\Phi_1^{(2i)}$, $i=1,\ldots,k-1$ so that the $k-1$ degree polynomial
 \begin{align}\eqlab{Pk_1}
     P_{k-1}(x_2) = \Phi_1^{(2)}+\cdots + \Phi_1^{(2(k-1))}x_2^{k-2}+x_2^{k-1},
 \end{align}
 has $k-1$ simple roots at the first $k-1$ integers:
 \begin{align}\eqlab{Pkroots}
 P_{k-1}(1)=\cdots =P_{k-1}(k-1)=0.
\end{align}
 
 Then fix $\phi$ as any regularization function. 
  Given $\phi'(y_{2c})>0$ and $y_{2c}=\phi^{-1}\left(\frac{-Y_-'}{Y_+'-Y_-'}\right) \in \mathbb R$, as well as $\Phi_1^{(2i)}$, $i=1,\ldots,k$, defined above, we proceed to define for each $\delta>0$ the function $\psi_k:\mathbb R \rightarrow \mathbb R$ as the polynomial of degree (at most) $2k$ with 
  \begin{align*}
 \psi_k(y_{2c})=\frac{-Y_-'}{Y_+'-Y_-'}, \quad \psi_k'(y_{2c})=\phi'(y_{2c}),\quad \psi_k^{(2i+1)}(y_{2c})=0\mbox{ for all } i=1,\ldots,k-1,
 \end{align*} 
  %\note{A suggestion: maybe in (4.6) it would be better to put (I am not sure, I will check it once more):\[\psi_k^{(2i)}(y_{2c}) = \Phi_0^{(2i)}+ \delta^{2(k-i)} {\phi'(y_{2c})^{2i-1}} C_{2i}^{-1}\Phi_1^{(2i)}(2i+1)!\] and here define $\Phi_0^{(2i)}:=-\frac{J_{2i+1}(\psi_k'(y_{2c}),\dots,\psi_k^{(2i-1)}(y_{2c}))\phi'(y_{2c})^{2i-1}}{C_{2i}}$. Then the values for $\psi_k^{(2i)}(y_{2c})$ are recursively well-defined. Then later using the fact that $\psi_k=\phi_k$ near $y_{2c}$, they have the same derivatives in $y_{2c}$ and then we use Lemma 4.2 with $\phi_k$ as already mentioned below  and  get the expression for $I(\delta x_2)$ before (4.10).}
 and where
 \begin{align}\eqlab{psik2i}
  \psi_k^{(2i)}(y_{2c}) = \Psi_0^{(2i)}+ (2i+1)! \delta^{2(k-i)} {\phi'(y_{2c})^{2i-1}} C_{2i}^{-1}\Phi_1^{(2i)},
\end{align}
for $i=1,\ldots,k$.  Here $\Psi_0^{(2i)}$, $i=1,\ldots,k$ are defined  recursively as the values of $\phi^{(2i)}(y_{2c})$ such that $I^{(2i+1)}(0)=0$:
\begin{align*}
  \Psi_0^{(2i)} = -C_{2i}^{-1} J_{2i-1}(\psi_k'(y_{2c}),\ldots,\psi_k^{(2i-1)}(y_{2c}))\phi'(y_{2c})^{2i-1},
  \end{align*}
recall \eqref{C2k} and \eqref{Psi0}. Then for each $\delta>0$ these $2k+1$ conditions fix the polynomial $\psi_k$ uniquely and
%  \end{align*}
from \lemmaref{I2k1} it follows that  \begin{align}
I^{(2i+1)}(0) = (2i+1)! \delta^{2(k-i)}\Psi_1^{(2i)}.\eqlab{I2i1here}
 \end{align}
if we replace $\phi$ by $\psi_k$ in the expression for $I(x)$.
%   Moreover, by Taylor's theorem
 
%   for some smooth $D$. 
\response{$\psi_k$ is, however, not a regularization function. Instead, we construct the regularization function $\phi_k$ by modifying $\phi$ such that it agrees with $\psi_k$ on a small neighborhood of $y_{2c}$ and, in particular, has the prescribed derivatives $\psi_k^{(i)}(y_{2c})$, $i=0,1,\ldots,2k$ at $y_2=y_{2c}$.} For this purpose, let $B:\mathbb R\rightarrow \mathbb R$ be a smooth ``bump function'' with support on $]-2,2[$ that is $1$ on the domain $[-1,1]$. Let $\upsilon>0$ and define $B_\upsilon(x) = B(\upsilon^{-1} x)$. Then $B_\upsilon$ is a bump function with support on $]-2\upsilon,2\upsilon[$ that is $1$ on the domain $[-\upsilon,\upsilon]$. Clearly, 
$$
\vert B_\upsilon'(x)\vert \le \upsilon^{-1}\text{sup}\vert B'\vert,$$ for all $x\in \mathbb R$. 
 We then define $\phi_k$  as follows:
 \begin{align}
  \phi_k(y_2):=\phi(y_2)(1-B_\upsilon(y_2-y_{2c}))+\psi_k(y_2) B_\upsilon(y_2-y_{2c}).\eqlab{barphi}
 \end{align}
Notice that $\phi_k(y_2) = \psi_k(y_2)$ on $[y_{2c}-\upsilon,y_{2c}+\upsilon]$ whereas $\phi_k(y_2) = \phi(y_2)$ outside $]y_{2c}-2\upsilon,y_{2c}+2\upsilon[$. It is clear that $\phi_k$ satisfies \ref{assA} and \ref{assC}. Therefore to verify that $\phi_k$ is a regularization function we just need to show \ref{assB}. By taking $\upsilon>0$ small enough, we have that 
  \begin{align}\eqlab{psicond}
  \psi_k'(y_2)>0 \mbox{ for } y_2\in [y_{2c}-\upsilon,y_{2c}+\upsilon].
  \end{align} 
  and consequently, it suffices to verify \ref{assB} on  $[y_{2c}-2\upsilon,y_{2c}-\upsilon]\cup [y_{2c}+\upsilon,y_{2c}+2\upsilon]$. We have
 \begin{equation}\eqlab{barphip}
 \begin{aligned}
  \phi_k'(y_2) &= \phi'(y_2) (1-B_\upsilon(y_2-y_{2c}))+\psi_k'(y_2)B_\upsilon(y_2-y_{2c})\\
  &+ (\psi_k(y_2) -\phi(y_2))B_\upsilon'(y_2-y_{2c}).
 \end{aligned}
 \end{equation}
 Seeing that $\phi'(y_{2c})=\psi_k'(y_{2c})$, the first two terms can be bounded by Taylor's theorem from below by $\frac12 \phi'(y_{2c})$ on the relevant domain by taking $\upsilon$ small enough. Similarly, using also that $\phi(y_{2c})=\psi_k(y_{2c})$ we have that 
\begin{align*}
 \psi_k(y_2) -\phi(y_2) = \frac12 D(y_2)(y_{2}-y_{2c})^2,
\end{align*}
for some smooth $D$. Hence on $[y_{2c}-2\upsilon,y_{2c}+2\upsilon]$ we have for all $\upsilon>0$ that
\begin{align*}
 \vert \psi_k(y_2) -\phi(y_2) \vert \le C \upsilon^2,
 \end{align*}
 for some constant $C>0$ independent of $\upsilon>0$. This allow us to bound the final term in \eqref{barphip} from below by $-C\text{sup}\vert B'\vert \upsilon$ for all $\upsilon>0$ small enough and consequently we have specifically shown that 
 \begin{align*}
\phi_k'(y_2)>0,
\end{align*}
for all $\upsilon>0$ small enough. Notice this holds uniformly for $\delta>0$ small enough.

We now apply \lemmaref{I2k1} with the regularization function $\phi_k$. Using \eqref{psik2i}, Taylor's theorem and setting $x=\delta x_2$ for $\delta>0$, we obtain the following expression
\begin{align*}
 I(x) = \delta^{2k+1} x_2^3 \left( P_{k-1}(x_2^2)+ \mathcal O(\delta)\right), 
\end{align*}
for the slow divergence-integral,
where $P_{k-1}$ is precisely the polynomial \eqref{Pk_1} of degree $k-1$. This is a simple calculation based upon \eqref{I2k1expr}, see also \eqref{I2i1here}. 
On $[y_{2c}-\upsilon,y_{2c}+\upsilon]$ where $\phi_k=\psi_k$ we then consider 
\begin{align}\eqlab{I2}
I_2(x_2,\delta)=\delta^{-2k-1} x_2^{-3} I(\delta x_2).
\end{align}By construction, recall \eqref{Pkroots}, we have \response{$I_2(x_2,0)=P_1(x_2)=0$} for each $x_2=1,\ldots,\sqrt{k-1}$,  and each root perturbs to a simple root of $I_2(\cdot,\delta)$ (and consequently a positive root of $I$ at $x\approx \delta,\ldots,\delta \sqrt{k-1}$) by the implicit function theorem for $\delta>0$ small enough. This completes the proof. 
\end{proof}

Suppose that 
\[
\beta>\chi>0,\quad \xi\ne 0.
\]
Then it is a simple calculation to show that assumptions \ref{assD}, \ref{assE} and \ref{assF} all hold true for $Z_\pm$ with $Z_+$ being quadratic of the form
\begin{align}\eqlab{Zp1}
    Z_+(z,\lambda) &= \begin{pmatrix}
                    1\\
                    x+\frac12 \xi x^2
                   \end{pmatrix},
% % \\
\end{align}
whereas $Z_-$ is linear of the form
\begin{align}\eqlab{Zn1}
Z_-(z,\lambda) &=\begin{pmatrix}
                   -\chi  \\
                   -\beta (x-\lambda)
                   \end{pmatrix}.
\end{align}
%This unfolding satisfies \eqref{versal} and hence (result on canard and lambda as breaking parameter?) holds true.
Upon invoking \thmref{mainslowdiv}, we then conclude that for each $k\in\mathbb N$ \thmref{thmhere} gives the existence of $k$ limit cycles of $z' = Z(z,\phi_k(y\epsilon^{-2}),\lambda^k_c(\epsilon))$ for all $0<\epsilon\ll 1$. In turn, this then completes the proof of \thmref{mainthm}.
\section{Numerical examples}\seclab{num}
To illustrate and quantify some of the aspects of our general approach, we consider \eqref{Zp1} and \eqref{Zn1} with
\begin{align}\eqlab{para1}
\beta=2,\,\chi=\xi=1,
\end{align} and use the general procedure in the proof of \thmref{mainthm} in \secref{final} to find three different $ \phi_k$-functions (tuning the parameters $\delta$ and $\upsilon$) so that $I$ has $3$, $5$ and $7$ simple zeros. We define our bump function $B$ in the following classical way. Let 
\begin{align*}
    B_0(x) = \begin{cases}
    0 & \text{for}\, x\le 0,\\
    e^{-1/x} &\text{for}\, x> 0,
    \end{cases}
\end{align*}
and put $B_1(x)=\frac{B_0(x)}{B_0(x)+B_0(1-x)}$, $B_2(x)=B_1(x-1)$, $B_3(x)=B_2(x^2)$, and then finally $B(x):=1-B_3(x)$.

For simplicity we use \begin{align*}
    \phi(y_2) = \frac12 + \frac{1}{\pi}\arctan(y_2),
\end{align*}
as our reference regularization function. Then with the parameters \eqref{para1} we find that
\begin{align*}
    y_{2c} =  \phi^{-1}\left(\frac{2}{3} \right) = \frac{1}{\sqrt{3}},\quad  \phi'(y_{2c}) =\frac{3}{4\pi}.
\end{align*}
For each $k=4$, $k=6$ and $k=8$ we then fix the polynomial $\psi_k$ by $\psi_k(y_{2c})=\frac23$, $\psi_k'(y_{2c})=\frac{3}{4\pi}$ and by setting $\psi_k^{(2i)}$, $i=1,\ldots,k$ equal to the expressions in \eqref{psik2i}; the quantities $\Phi_1^{(2i)}$, $i=1,\ldots,k$ in \eqref{psik2i} are chosen so that $P_{k-1}$ \eqref{Pk_1} has its roots at $1,\ldots,k-1$. As outlined above we set all the odd higher order derivatives $\psi_k^{(2i+1)}(y_{2c})=0$, $i\ge 1$.  For $k=4$ with $\delta=10^{-3}$ we obtain
 \begin{equation}\begin{aligned}
    \psi_4(y_2) &=\frac{2}{3}+ \frac{3}{4\pi}\left(y_2 - \frac{1}{\sqrt{3}}\right)+0.2137243716\left(y_2- \frac{1}{\sqrt{3}}\right)^2+0.306956879\left(y_2- \frac{1}{\sqrt{3}}\right)^4\\
    &+1.442372260\left(y_2- \frac{1}{\sqrt{3}}\right)^6- 25.33517649\left(y_2- \frac{1}{\sqrt{3}}\right)^8.
\end{aligned}\eqlab{phik4}
\end{equation}
The resulting $ \phi_4$ \eqref{barphi}$_{k=4}$ is shown in \figref{Kphi}(a) for $\upsilon=0.05$ (in red) together with $ \phi$ (blue) and $\psi_4$ (green). For this $\phi_4$ we then proceed to accurately compute the slow divergence integral (using Taylor expansions up to terms of order $x^{25}$ computed in MAPLE with Digits set to $100$). The result is shown in \figref{KI}(a). Here $\upsilon>0$ is fixed so that $\phi_4$ satisfies \ref{assA}-\ref{assC}, whereas the value of $\delta$ is taken small enough to ensure that $I$ has (at least) $3$ simple zeros. We use the same method for $k=6$ and $k=8$ and find
\begin{equation}
\begin{aligned}
    \psi_6(y_2)& =\frac{2}{3}+ \frac{3}{4\pi}\left(y_2 - \frac{1}{\sqrt{3}}\right) +0.2137243716\left(y_2- \frac{1}{\sqrt{3}}\right)^2-0.3069568794\left(y_2- \frac{1}{\sqrt{3}}\right)^4\\
    &+1.44235445\left(y_2- \frac{1}{\sqrt{3}}\right)^6- 12.12351865\left(y_2- \frac{1}{\sqrt{3}}\right)^8\\
    &+154.2008391\left(y_2- \frac{1}{\sqrt{3}}\right)^{10}- 3015.15236\left(y_2- \frac{1}{\sqrt{3}}\right)^{12},
\end{aligned}\eqlab{phik6}
\end{equation}
and
 \begin{equation}\begin{aligned}
    \psi_8(y_2)& =\frac{2}{3}+ \frac{3}{4\pi}\left(y_2 - \frac{1}{\sqrt{3}}\right) +0.2137243716\left(y_2- \frac{1}{\sqrt{3}}\right)^2- 0.3069568794\left(y_2- \frac{1}{\sqrt{3}}\right)^4\\
    &+1.442354453\left(y_2- \frac{1}{\sqrt{3}}\right)^6- 12.12351865\left(y_2- \frac{1}{\sqrt{3}}\right)^8\\
    &+154.2008302\left(y_2- \frac{1}{\sqrt{3}}\right)^{10}- 2744.019283\left(y_2- \frac{1}{\sqrt{3}}\right)^{12}+ 65135.03549\left(y_2- \frac{1}{\sqrt{3}}\right)^{14}\\
    &- 1.998886089 \times 10^6\left(y_2- \frac{1}{\sqrt{3}}\right)^{16}
\end{aligned}\eqlab{phik8}
\end{equation}
for
$\delta=10^{-4}$ resp. $\delta=10^{-5}$. The result is shown in \figref{KI}(b) resp. (c), still with $\upsilon=0.05$. 

The roots are very sensitive with respect to $\delta$; increasing $\delta$ only slightly in each of our cases $k=4,6$, and $8$ lead to fewer roots. For example for $k=8$ we only find $5$ roots for $\delta\gtrsim 9.449 \times 10^{-5}$. In any case, $\delta>0$ has to be taken quite small to realize the desired number of roots. In turn, this implies that $I$ is extremely small on the relevant domains; for $k=8$ for example, we find (see \figref{KI}(c)) that $I(x)\sim 10^{-82}-10^{-85}$! We therefore expect -- in line with \cite{de2013a} -- that the desired number of limit cycles for $0<\epsilon\ll 1$ can also only be realized for extremely small values of $\epsilon>0$ and that these are therefore extremely difficult (if not impossible) to detect in numerical computations.

\begin{figure}[h!]
\begin{center}
\subfigure[]{\includegraphics[width=.45\textwidth]{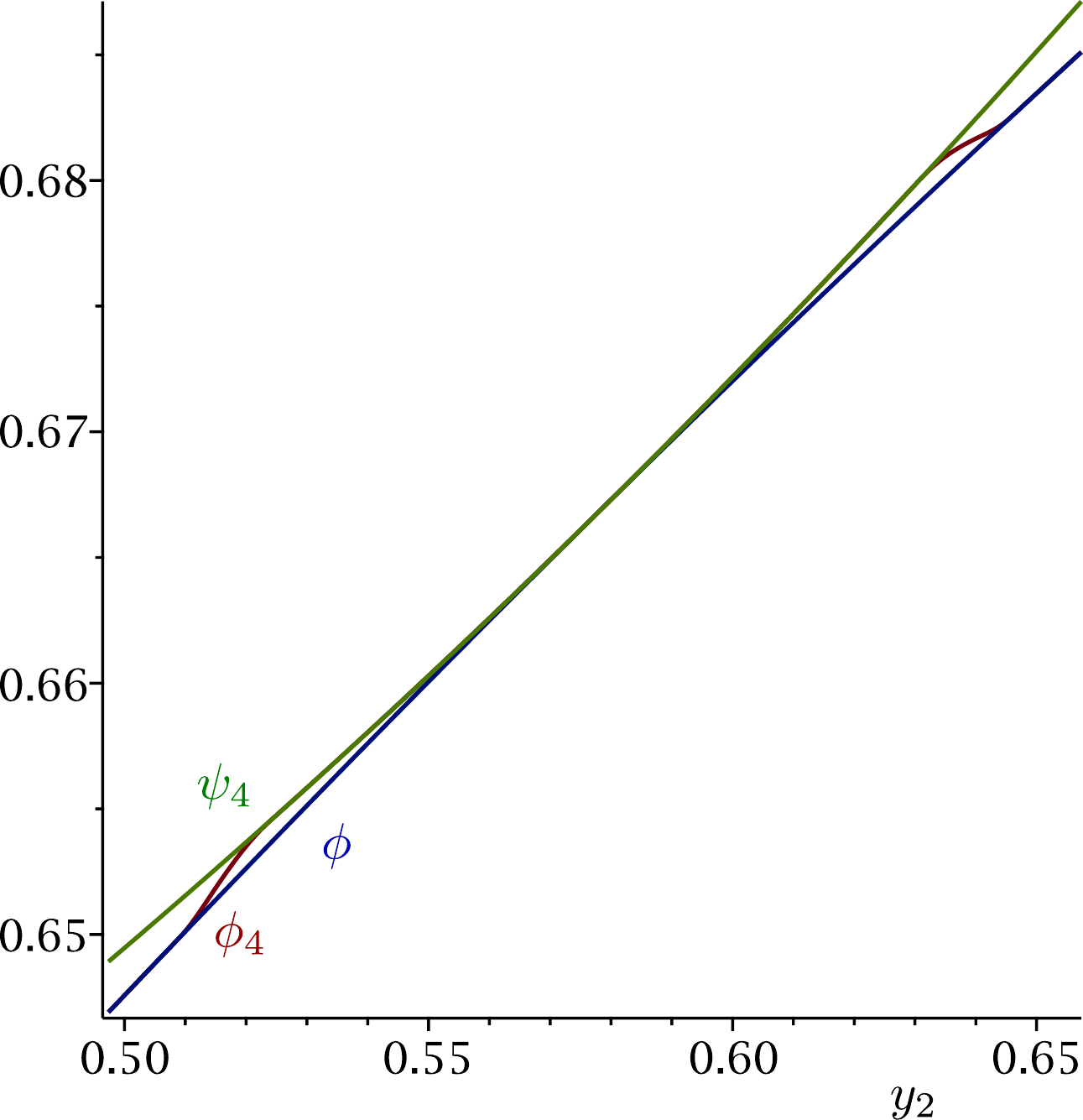}}
\subfigure[]{\includegraphics[width=.45\textwidth]{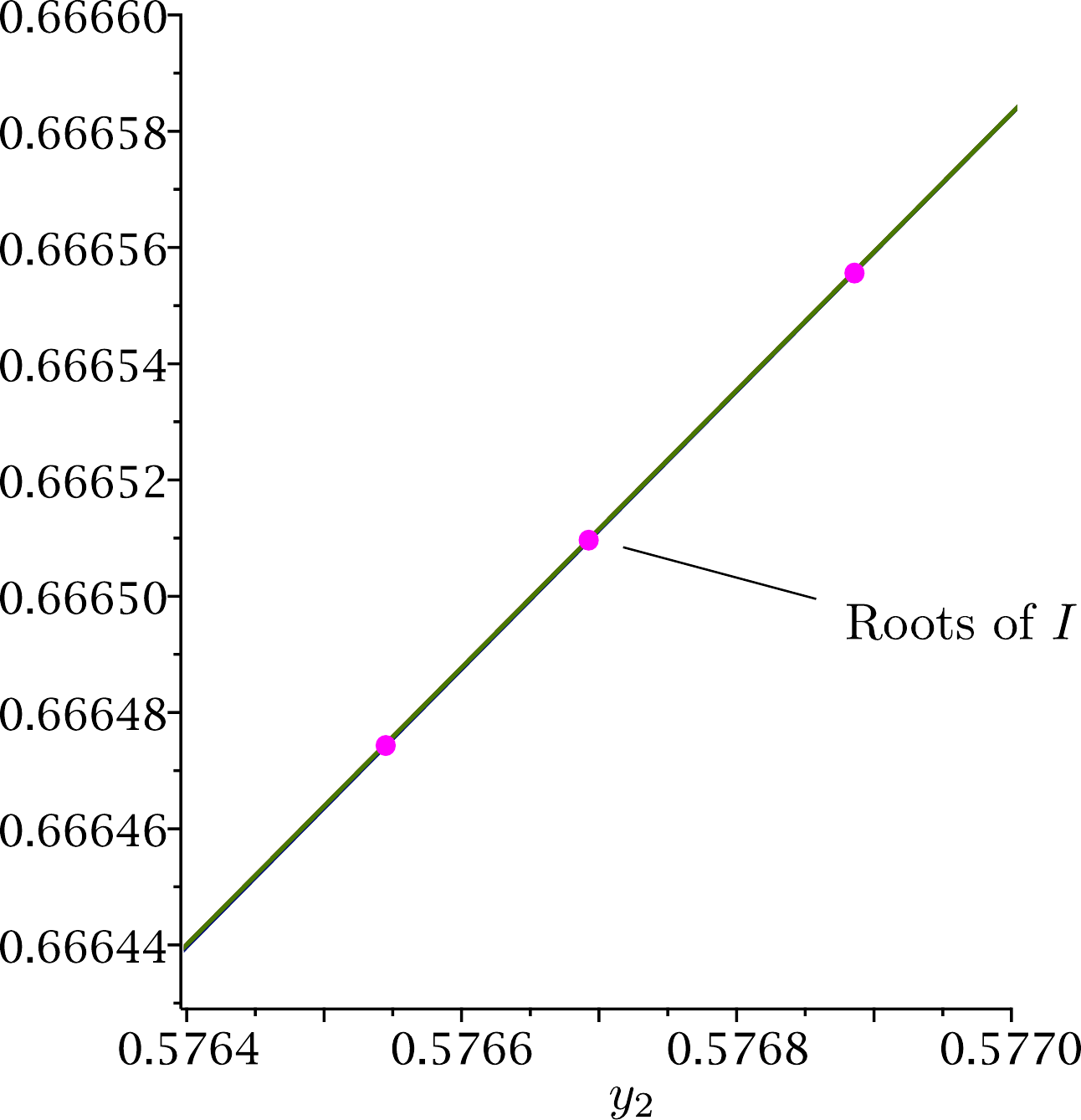}}
% \subfigure[]{\includegraphics[width=.495\textwidth]{layerReducedBlowup.pdf}}
% \input{posusts}
\end{center}
\caption{In (a): Graph of the regularization function $\phi_4$ (in red) used for generating an example with $3$ simple zeros of $I$, see \figref{KI}(a). The function $\phi$ is in blue whereas $\psi_4$ is in green. In (b): A zoom around $y_2=y_{2c}$ showing the three values of $y_2$ on the critical manifold corresponding to the $x$-point where $I(x)=0$. Notice that these points lie inside the region where $\phi_4=\psi_4$.}
\figlab{Kphi}
\end{figure}

\begin{figure}[h!]
\begin{center}
\subfigure[]{\includegraphics[width=.42\textwidth]{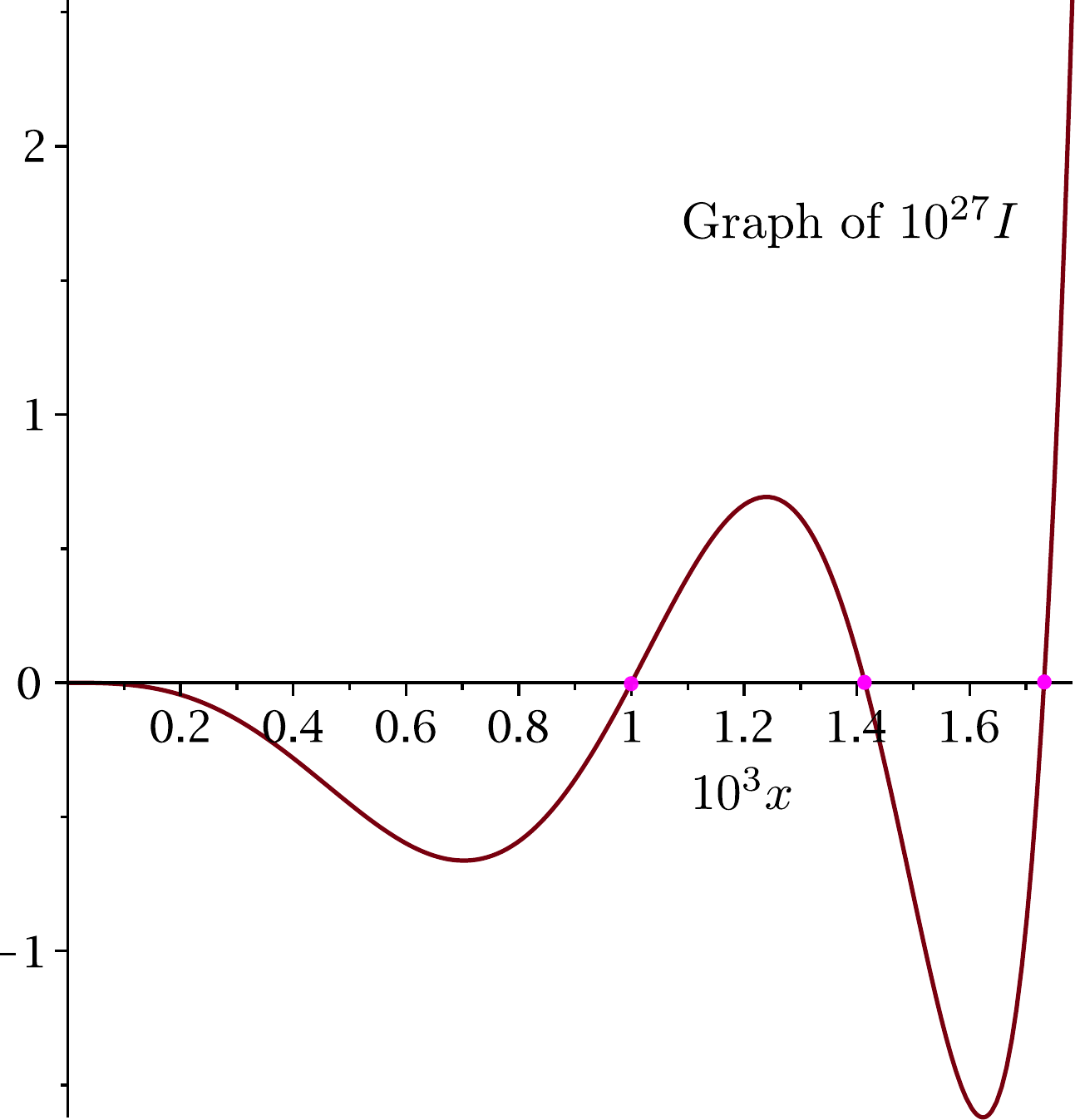}}
\subfigure[]{\includegraphics[width=.42\textwidth]{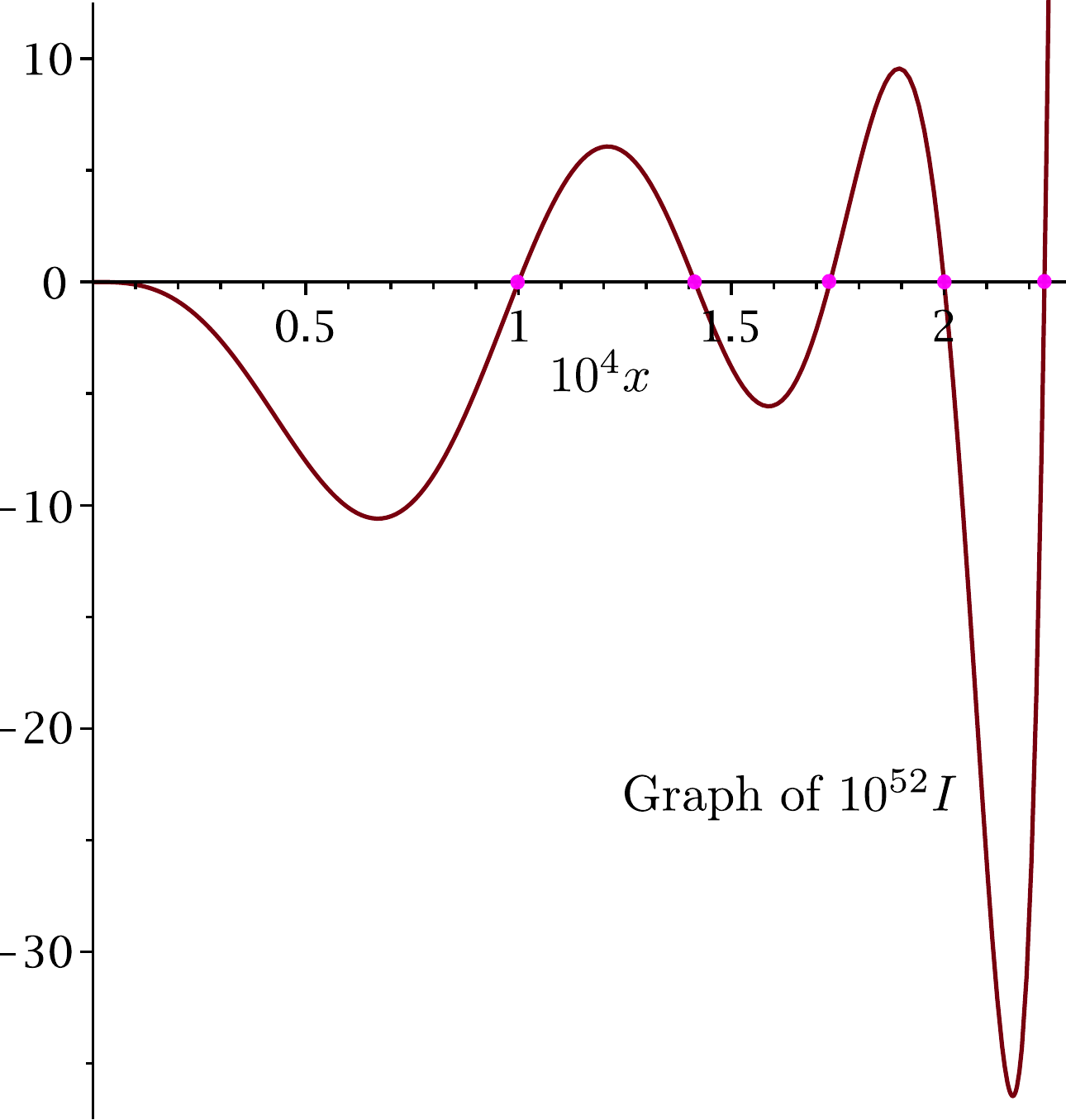}}
\subfigure[]{\includegraphics[width=.42\textwidth]{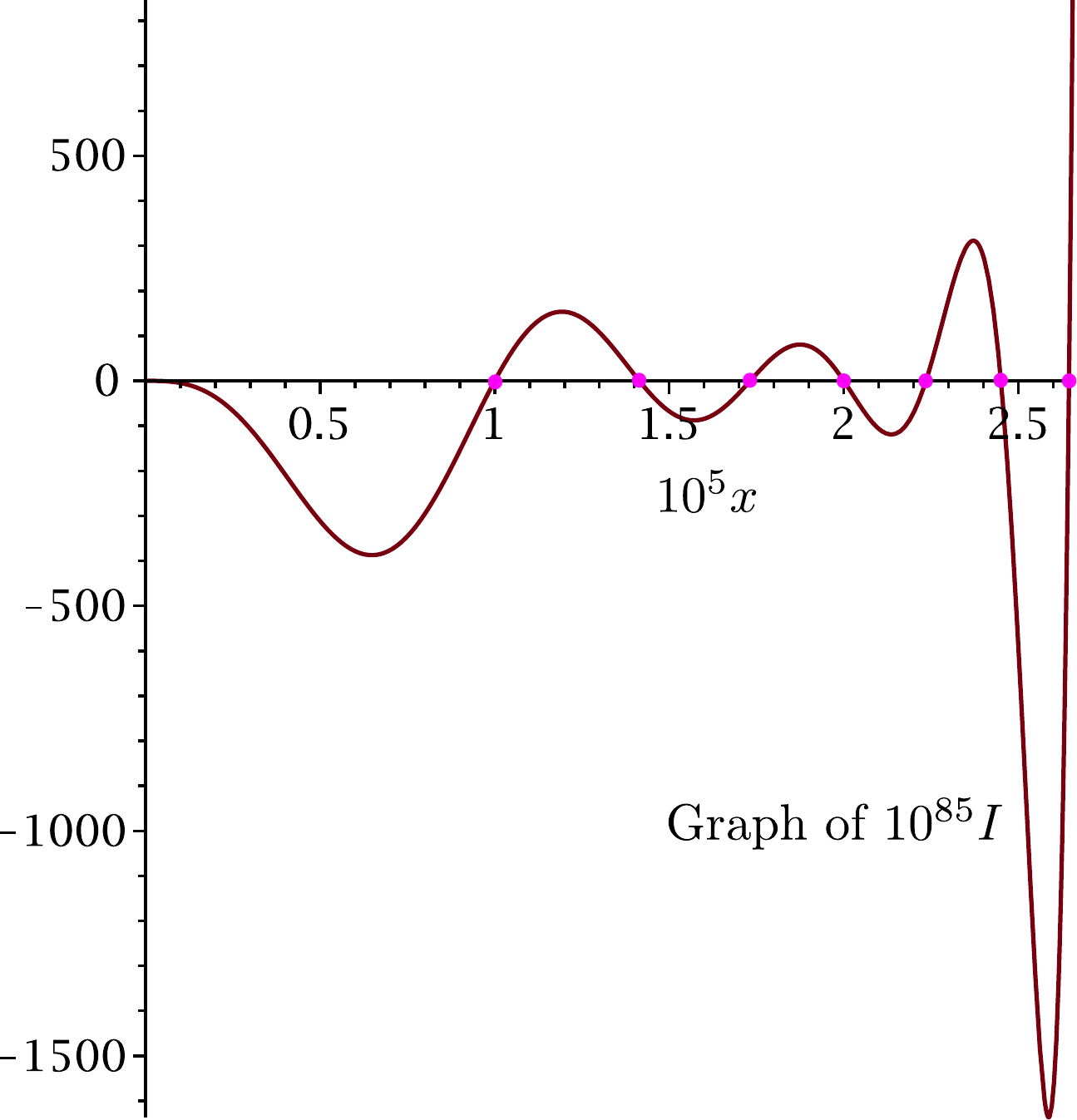}}
% \subfigure[]{\includegraphics[width=.495\textwidth]{layerReducedBlowup.pdf}}
% \input{posusts}
\end{center}
\caption{The graphs of (scaled versions of) the slow divergence integral $I$, see \eqref{I2}, for $\psi_k$ given by \eqref{phik4} (a), \eqref{phik6} (b), \eqref{phik8} (c). Notice that in each case, $I$ has roots close to $\delta \sqrt{i}$, $i=1,\ldots,k-1$ as desired. }
\figlab{KI}
\end{figure}

\appendix

\section{Blowing up the degenerate line $\overline{H}$}\label{appendix-0}
In this section we show that the system \eqref{global} from Section \ref{section-primary-fam} satisfies Assumptions T3--T6 of \cite{DM-entryexit} near the intersection $p_0$ of $\overline H$ with $\overline C$. We write 
\begin{equation}\label{Taylor}
Z_\pm(\cdot,\lambda_0+\epsilon \widetilde\lambda)=Z_\pm (\cdot,\lambda_0)+\epsilon \widetilde\lambda\widetilde Z_\pm(\cdot)+\mathcal O(\epsilon^2)
\end{equation}
where $\widetilde Z_\pm=(\widetilde X_\pm,\widetilde Y_\pm)$, like in \secref{section-2-fold}, and $\widetilde\lambda\sim 0$ is introduced in \secref{cylblowup}.
We blow up $\overline{H}$ to a cylinder through the following blow-up transformation
 \begin{align}\label{bup-second}
  \rho\ge 0,\,(\tilde x,\tilde \epsilon) \mapsto \begin{cases}
                                             x&=\rho \tilde x,\nonumber\\
                                             \epsilon &=\rho \tilde \epsilon,
                                             \end{cases}                                              \end{align}
where $(\tilde x,\tilde \epsilon)\in\mathbb{S}^1$ and $\tilde\epsilon\ge 0$. Again we will work with different charts. Let's first consider the end points of the normally attracting part and the normally repelling part of $\overline C$ on the edge  of the cylinder $\{\rho=\tilde\epsilon=0\}$.

\subsection{Dynamics in the phase directional charts $\tilde x=\pm 1$}\label{section-secondary-phase}
In the phase directional chart $\tilde{x}=1$ we have $(x,\epsilon)=(\rho,\rho\tilde{\epsilon})$. In these coordinates system \eqref{global} becomes, after division by $\rho>0$,
\begin{equation}\label{phase-second}
\begin{aligned}
 \dot y_2 &= Y_+'(0,0,\lambda_0)\phi(y_2)+Y_-'(0,0,\lambda_0)(1-\phi(y_2))+\mathcal O(\rho,\tilde\epsilon),\\
 \dot \rho&=\rho\tilde{\epsilon}^2\left( X_+(0,0,\lambda_0)\phi(y_2)+X_-(0,0,\lambda_0)(1-\phi(y_2))+\mathcal O(\rho)  \right),\\
 \dot{\tilde\epsilon} &= -\tilde\epsilon^3\left( X_+(0,0,\lambda_0)\phi(y_2)+X_-(0,0,\lambda_0)(1-\phi(y_2))+\mathcal O(\rho)  \right).
\end{aligned}
\end{equation}
When $\rho=\tilde\epsilon=0$, system (\ref{phase-second}) has a semi-hyperbolic singularity \[y_2=y_{2c}:=\phi^{-1}\left(\frac{-Y_-'}{Y_+'-Y_-'}(0,0,\lambda_0)\right).\] The eigenvalues of the linearization are \[((Y_+'-Y_-')(0,0,\lambda_0)\phi'(y_2),0,0),\]
the first eigenvalue being positive by  \propref{VI3} item (iv). Two-dimensional center manifolds of (\ref{phase-second}) at this singularity are transverse to the unstable manifold given by the $y_2$-axis. Thus, the end point of the repelling part of $\overline C$ is normally hyperbolic (Assumption T3). Moreover, each center manifold of (\ref{phase-second}) at the singularity is the graph of %\[y_2=\phi^{-1}\left(\frac{-Y_-'}{Y_+'-Y_-'}(0,0,\lambda_0)\right)+\mathcal O(\rho,\tilde\epsilon).\]
\[y_2=y_{2c}+\mathcal O(\rho,\tilde\epsilon).\]
Using the $(\rho,\tilde\epsilon)$-component of (\ref{phase-second}) we easily find the center behavior: \[\left\{\dot\rho=\rho\tilde\epsilon^2\left(X_{sl}(0,\lambda_0)+\mathcal O(\rho,\tilde\epsilon) \right),\dot{\tilde\epsilon}=-\tilde\epsilon^3\left(X_{sl}(0,\lambda_0)+\mathcal O(\rho,\tilde\epsilon) \right)\right\}.\]
Since $X_{sl}(0,\lambda_0)>0$, this system has, after division by $\tilde\epsilon^2$, an isolated hyperbolic saddle $(\rho,\tilde\epsilon)=(0,0)$ (Assumption T4). Notice that the exponent in $\tilde\epsilon^2$ is equal to the order of degeneracy mentioned in Section \ref{section-primary-fam}. Notice also that the center manifold, restricted to $\rho=0$, is unique because (\ref{phase-second}) is of saddle type inside $\rho=0$.

The chart $\tilde x=-1$ can be covered by applying $(t,\rho,\tilde{\epsilon})\mapsto (-t,-\rho,-\tilde{\epsilon})$ to (\ref{phase-second}).

\begin{remark}
In the framework of \cite{DM-entryexit} a turning point is usually replaced with a sphere $\mathbb{S}^2$ and Assumptions T3-T4 have to be satisfied at the end points of normally hyperbolic branches of the critical curve on the equator of the sphere.  In our slow-fast setting \eqref{global} it is more convenient to use a cylindrical blow-up. This is not a problem because locally near the end points, located on the edge of the cylinder, one can use the normal linearization theorem of \cite{Bonk}, like in \cite{DM-entryexit}. 
 \end{remark}

\subsection{Dynamics in the family chart $\tilde\epsilon=1$ }\label{section-secondary-family}
In the scaling chart we obtain $x=\epsilon x_2$. The system \eqref{global} changes into
\begin{equation}\eqlab{x2y2}
\begin{aligned}
 \dot x_2 &= X_+(0,0,\lambda_0)\phi(y_2)+X_-(0,0,\lambda_0)(1-\phi(y_2)),\\
 \dot y_2 &=\left(x_2Y_+'(0,0,\lambda_0)+\widetilde\lambda\widetilde{Y}_+(0,0)\right)\phi(y_2)+\left(x_2Y_-'(0,0,\lambda_0)+\widetilde\lambda\widetilde{Y}_-(0,0)\right)(1-\phi(y_2)),
\end{aligned}
\end{equation}
upon desingularization (dividing the right hand side by $\epsilon$)
and (subsequently) setting $\epsilon=0$. We used (\ref{Taylor}). For $\widetilde\lambda=0$, \eqref{x2y2} has an invariant line $\gamma$ defined by $y_2=y_{2c}$
% \begin{align}
 % y_2=\phi^{-1}\left(\frac{-Y_-'}{Y_+'-Y_-'}(0,0,\lambda_0)\right),\nonumber\eqlab{line}
 %\end{align}
with the dynamics $\dot{x}_2=X_{sl}(0,\lambda_0)$ on it. The line $\gamma$ is a heteroclinic connection on the cylinder connecting the end point of the attracting part of $\overline{C}$ to the end point of the repelling part of $\overline C$ (Assumption T5). See also \figref{blowup}(b).

To show that the invariant line $\gamma$ breaks in a regular way as we vary $\widetilde \lambda\sim 0$ (Assumption T6), we follow \cite[section 6.2]{kristiansen2015a} and extend the system by augmenting $\dot{\widetilde \lambda}=0$. In this formulation the center manifolds from $\tilde x=\pm 1$ -- that extend the attracting and repelling parts of $\overline C$ onto $\overline H$ -- become two-dimensional and $\gamma$ belongs to the intersection of these within $\rho=0$ (where the manifolds are overflowing and unique). Write $\{\mbox{\eqref{x2y2}}, \dot {\widetilde\lambda}=0\}$ in terms of $\frac{dy_2}{dx_2},\frac{d\widetilde \lambda}{dx_2}$ and consider the variational equations around the solution $y_2=y_{2c}$, $\widetilde \lambda=0$ (corresponding to $\gamma$):
\begin{equation}\eqlab{var}
\begin{aligned}
\frac{du}{dx_2} &= Ax_2 u + B v,\\
\frac{dv}{dx_2}&=0,
\end{aligned}
\end{equation}
where
\begin{align*}
    A := \frac{Y_+'-Y_-' }{X_{sl}}(0,0,\lambda_0)\phi'(y_{2c}),\quad B:= \frac{\widetilde Y_-Y_+'-\widetilde Y_+Y_-'}{X_{sl}(Y'_+-Y'_-)} (0,0,\lambda_0).
\end{align*}
Notice that $A>0$, see \propref{VI3} item (iv), and that $B\ne 0$ by \ref{assD}, see specifically \eqref{versal}. It is then straightforward to show, using the asymptotics of the error-function $\text{erf}$, see also \cite[Lemma 6.2]{kristiansen2015a}, that there are two linearly independent solutions of \eqref{var}:
\begin{align*}
(u,v) &= \left(B \sqrt{\frac{2\pi }{A}}e^{Ax_2^2/2}\left(\text{erf}\left(\sqrt{\frac{A}{2}}x_2\right)\pm 1\right),1\right),%\\
%(u,v) &= \left(B \sqrt{\frac{2\pi }{A}}e^{Ax_2^2/2}\left(\text{erf}\left(\sqrt{\frac{A}{2}}x_2\right)-1\right),1\right),\\
\end{align*}
with exponential growth for $x_2\rightarrow \infty$ (resp. $-\infty$) and algebraic growth for $x_2\rightarrow -\infty$ (resp. $\infty$). By \cite[Proposition 4.2]{szmolyan_canards_2001} the extended center manifolds therefore intersect transversally along $\gamma$, which completes the verification of Asssumption T6. 
%Following \secref{section-breaking}, the invariant line $\gamma$ breaks in a
%regular way as we vary $\widetilde\lambda\sim 0$ .

\section{Transition maps near the hyperbolic edges}\label{appendix}
In this section we study the transition map near the line of singularities $\{r_1=\epsilon_1=0\}$ of 
\begin{equation}\eqlab{Z1}
\begin{aligned}
    \dot x &= r_1^2 X(x,r_1,\epsilon_1),\\
    \dot r_1 &=-r_1 Y(x,r_1,\epsilon_1),\\
    \dot \epsilon_1 &=\epsilon_1 Y(x,r_1,\epsilon_1),
\end{aligned}
\end{equation}
where $X$ and $Y$ are smooth functions. We assume that $Y(x,0,0)>0$ for each $x\in J$, $J$ being a compact set. Notice that $(x,0,0)$ is a set of equilibria and the linearization has $\mp Y(x,0,0)$ as two nonzero eigenvalues.

We now describe a transition map near this line of partially hyperbolic equilibria. We consider the transition map $Q_1$ from $\Sigma_{in}:=\{(x,r_1,\epsilon_1): r_1=r_{10}\}$ to $\Sigma_{out}:=\{(x,r_1,\epsilon_1): \epsilon_1=\epsilon_{10}\}$ along the trajectories of \eqref{Z1} where $r_{10},\epsilon_{10}$ are small positive constants. Let $\pi_x Q_1$ denote the $x$-component of $Q_1$.
\begin{proposition}\label{prop-passage-hyperbolic}
Fix $n \in \mathbb N$ then there are constants $r_{10}>0$ and $\epsilon_{10}>0$ small enough such that  \begin{align*}
x\mapsto 
    \pi_x Q_1(x,r_{10},\epsilon_{1})
\end{align*}
is $C^n$ uniformly and continuously in $\epsilon_1$. In particular, 
\begin{align*}
\pi_x Q_1(x,r_{10},\epsilon_{1})=g_0(x) + \mathcal O(\epsilon_1 \log \epsilon_1^{-1}), \ \epsilon_1\to 0,
\end{align*}
with $g_0$ smooth
and this expression can be differentiated with respect to $x$ without changing the order of the remainder.
\end{proposition}
\begin{proof}
We work with the equivalent field \eqref{Z1} divided by $Y>0$ on $J\times [0,r_{10}]\times [0,\epsilon_{10}]$. We denote this vector field by $\widetilde F$. First we prove the following lemma.
\begin{lemma}
For $r_{10}$ sufficiently small, there exists a diffeomorphism 
 %\begin{align*}
  %   \Phi:I\times [0,\delta]\times [0,\nu]\ni (x,r_1,\epsilon_1)\mapsto \Phi(x,r_1,\epsilon_1)
 %\end{align*}
$\Phi$
\begin{align*}
    \Phi(x,r_1,\epsilon_1) = \begin{pmatrix}
    h(x,r_1)\\
    r_1\\
    \epsilon_1\end{pmatrix},
\end{align*}
with $h_x(x,r_1)\ne 0$ for all $x\in J,\,r_1\in [0,r_{10}]$,
such that $\widetilde F$ changes into
\begin{equation}\eqlab{xieqn}
\begin{aligned}
    \dot \xi &= r_1^2\epsilon_1 G(\xi,r_1,\epsilon_1),\\
    \dot r_1 &=-r_1,\\
    \dot \epsilon_1 &=\epsilon_1,
\end{aligned}
\end{equation}
for some smooth $G$.
\end{lemma}
\begin{proof}
The map $\Phi$ is obtained by straightening out the stable manifolds of  points $(\xi,0,0)$. These manifolds are contained within $\epsilon_1=0$ and are graphs over $r_1$:
\begin{align*}
x = g(\xi,r_1)
\end{align*}
In particular, $g(\xi,0)=\xi$, $g_\xi(\xi,0)\ne 0$ and we can invert this expression for $\xi$:
\begin{align*}
    \xi = h(x,r_1)
\end{align*}
with $h(\cdot,r_1)=g^{-1}(\cdot,r_1)$. Seeing that $\dot \xi=0$ for $\epsilon_1=0$ we obtain the result by smoothness of the right hand side. 
\end{proof}
We then proceed to work on the normal form \eqref{xieqn}, describing the transition map $Q_1$ from $\Sigma_{in}=\{(\xi,r_1,\epsilon_1):r_{1}=r_{10}\}$ to $\Sigma_{out}=\{(\xi,r_1,\epsilon_1):\epsilon_1=\epsilon_{10}\}$. Let $\pi_\xi Q_1$ denote the $\xi$-component. 

First we realize that $r_1\epsilon_1=const.$ is conserved. 
Integrating the last two equations from $r_1(0)=r_{10}$, $\epsilon_1(0)=``\epsilon_{1}"$ and inserting this into the first one, we obtain 
\begin{align}
    \xi(T) = \xi(0) + r_{10} \epsilon_{1} \int_0^Te^{-s} r_{10} G(\xi(s),e^{-s} r_{10},e^{s}\epsilon_{1}) ds,\eqlab{xiT}
\end{align}
where the transition time $T=\log (\epsilon_1^{-1} \epsilon_{10})$. From here we directly obtain that 
\begin{align*}\pi_\xi Q_1(\xi,r_{10},\epsilon_{1}) =\xi+\mathcal O(\epsilon_1 \log \epsilon_1^{-1}), \ \epsilon_1\to 0
\end{align*}
because the integrand in \eqref{xiT} is bounded on the segment $[0,T]$. %on the domain $J\times [0,r_{10}]\times [0,\epsilon_{10}]$ (in principle we have to enlarge $I$ slightly first, since $\xi(s)$ with $\xi(0)\in I$ may leave $I$ in $s\in (0,T)$; however, taking $J$ compact such that $I\subset J$ then there is a $r_{10}$ small enough such that all $\xi(0)\in I$ remain within $J$ for $s\in (0,T)$ for all $r_1\in [0,r_{10}]$ and we denote by $C$ the upper bound of $\epsilon_{10} G$ on $J\times [0,r_{10},\epsilon_{10}]$). 
We handle the derivatives of $\pi_\xi Q_1$ with respect to $\xi$ in a similar way by considering the higher variational equations of \eqref{xieqn}. We skip the details because it is standard, see e.g. \cite{Gucwa2009783} or \cite[Proposition 3.3]{ZhangEdge}.
\end{proof}
If $X$ and $Y$ in \eqref{Z1} depend smoothly on a parameter $\alpha$, then $\pi_x Q_1$ will also depend $C^n$-smoothly on this parameter. This also follows from studying \eqref{xieqn}. We simply study the variational equations obtained by differentiating with respect to $\alpha$ and apply a similar estimation.

%Consequently we have
%\begin{cor}
%The number of limit cycles for regularization of quadratic piecewise systems is unbounded. 
%\end{cor}
% \begin{align

% We proceed in the usual way \cite{}. 
% Maybe the main theorem should be stated like this (\textbf{yes, I agree}):
% \begin{theorem}
% Suppose that $Z_\pm$ has a two-fold at $z=0$ for $\alpha=0$ and that $\alpha\ne 0$ unfolds the two-fold (nondegenericity condition). Suppose also that assumption \ref{assE} and E holds true and fix a compact neighborhood $U$ of $z=0$. Then the following holds true for any $k\in \mathbb N$:
% 
% There exists an $\epsilon_0>0$, a function $\alpha_c:[0,\epsilon_0]\rightarrow \mathbb R$ and a regularization function $\phi_k$ satisfying (A)-(C) such that the regularization $Z$ of $Z_\pm$ for $\alpha=\alpha_c(\epsilon)$ has at least $k$ limit cycles contained in $U$ for all $0<\epsilon\le \epsilon_0$.
% \end{theorem}
% % \section{The number of limit cycles of a regularization of a PWS system is unbounded}
% % \begin{theorem}
% %  Consider $Z_\pm$ 
% % \end{theorem}
% 
% 
% % We consider a regularization of PWS system
% % \begin{align*}
% %  \dot z = Z(z,\phi(y\epsilon^{-1},\alpha),
% % \end{align*}
% % with $Z(z,p,\alpha)=Z_+(z,\alpha)p+Z_-(z,\alpha)(1-p)$. We assume that the PWS system has a two-fold at $z=0$ for $\alpha=0$.

% We will now  $\phi(y_{2c}) = 

% It follows that once we assume \text{E} then $I^{(3)}(0)$ can be either 
%Here we have used the symmetry assumption to simplify the expression. 
% \begin{align*}
%  Z_+(z,p) = \begin{pmatrix}
%              
%             \end{pmatrix}
% 
% \end{align*}
\bibliography{refs}
\bibliographystyle{plain}
% \newpage
% \input{reply}
\end{document}